\documentclass[a4paper,12pt,twoside]{article}

\usepackage{ucs}
\usepackage[utf8x]{inputenc}
\usepackage{amsmath}
\usepackage{amsfonts}
\usepackage{amssymb}
\usepackage{mathtools}
\usepackage{amsthm}
\usepackage[american]{babel}
\usepackage[T1]{fontenc}
\usepackage{graphicx}
\usepackage{enumerate}
\usepackage{hyperref}
\hypersetup{colorlinks=true,linkcolor=black,urlcolor=black,citecolor=black}
\usepackage[all]{xy}
\usepackage{multicol}
\usepackage{authblk}

\usepackage{geometry}
\numberwithin{equation}{section}

\begin{document}

\newtheorem{theorem}{Theorem}[section]
\newtheorem*{theorem*}{Theorem}
\newtheorem{proposition}[theorem]{Proposition}
\newtheorem{corollary}[theorem]{Corollary}
\newtheorem{lemma}[theorem]{Lemma}

\numberwithin{theorem}{section}

\theoremstyle{definition}
\newtheorem{definition}[theorem]{Definition}
\newtheorem{example}[theorem]{Example}
\newtheorem{examples}[theorem]{Examples}
\newtheorem{remark}[theorem]{Remark}
\newtheorem{notation}[theorem]{Notation}

\date{}
\title{Basic Kirwan Surjectivity for K-Contact Manifolds}
\author{Lana Casselmann\thanks{lana.casselmann@uni-hamburg.de}}
\affil{\small{GRK 1670\\ University of Hamburg\\ Bundesstra{\ss}e 55\\
20146 Hamburg, Germany}}

\newcommand{\R}{\mathbb{R}}
\newcommand{\C}{\mathbb{C}}
\newcommand{\g}{\mathfrak{g}}
\newcommand{\X}{\mathfrak{X}}
\newcommand{\CF}{\mathcal{F}}
\newcommand\numberthis{\addtocounter{equation}{1}\tag{\theequation}}

\maketitle

\vspace{-9pt}
\begin{abstract}
We prove an analogue of Kirwan surjectivity in the setting of equivariant basic cohomology of $K$-contact manifolds. If the Reeb vector field induces a free $S^1$-action, the $S^1$-quotient is a symplectic manifold and our result reproduces Kirwan's surjectivity for these symplectic manifolds. We further prove a Tolman-Weitsman type description of the kernel of the basic Kirwan map for $S^1$-actions and show that torus actions on a $K$-contact manifold that preserve the contact form and admit 0 as a regular value of the contact moment map are equivariantly formal in the basic setting.
\begin{center} The final publication is available at Springer via\\ \href{http://dx.doi.org/10.1007/s10455-017-9552-6}{http://dx.doi.org/10.1007/s10455-017-9552-6}.
\end{center}
\end{abstract}
\vspace{-10pt}

\tableofcontents

\subsubsection*{Acknowledgments}\vspace{-5pt}
The author would like to thank Oliver Goertsches for suggesting the topic of this paper, numerous helpful discussions, his careful guidance and critical reading of the manuscript, as well as Jonathan Fisher for helpful conversations. The author would further like to thank the anonymous reviewers for their careful reading and constructive comments on the paper. This research was supported by the RTG 1670 ``Mathematics inspired by String Theory and Quantum Field Theory'', funded by the Deutsche Forschungsgemeinschaft (DFG).

\section{Introduction}
The well known Kirwan surjectivity asserts that if $\Psi$ is a moment map for a Hamiltonian action of a compact group $G$ on a compact symplectic manifold $N$ and $0$ a regular value thereof, then the Kirwan map $H^*_G(N) \rightarrow H^*_G(\Psi^{-1}(0))$ induced by the inclusion is an epimorphism (\cite{kirwan1984cohomology}). The analogous statement for contact manifolds is known to no longer hold, as the following example by Lerman shows (cf. \cite{lerman04question}, where, apart from giving the counterexample, he states that the question what the kernel and cokernel of the induced map are in the contact case are still of interest).
\begin{example}
	Consider the 3-sphere $S^3= \{z \in \C^2 \mid |z_1|^2+|z_2|^2 =1\} \subset \C^2$ with the $S^1$-action defined by $\lambda \cdot (z_1, z_2) = (\lambda z_1,\lambda^{-1} z_2)$. Then $X(z_1, z_2)=(iz_1, -iz_2)$ is the fundamental vector field of $1 \in \R \simeq \mathfrak{s^1}$. Considering the ($S^1$-invariant) contact form $\alpha = \tfrac{i}{2}\sum_{j=1}^2(z_j d\bar z_j- \bar z_j d z_j )$, we compute the contact moment map (see Section \ref{subsec cont mom map}) to be 
	\[\Psi \colon S^3 \rightarrow \R, \qquad (z_1,z_2) \mapsto |z_1|^2 - |z_2|^2.\]
	Since the $S^1$-action is free, the equivariant cohomology is simply the ordinary cohomology of the $S^1$-quotient and we compute 
	\[ H^*_{S^1}(S^3)=H^*(\C P^1), \qquad H^*_{S^1}(\Psi^{-1}(0))=H^*(S^1).\]
	But there cannot exist an epimorphism from $H^*(\C P^1)$ to $H^*(S^1)$.
\end{example}
This motivates the search for a modification of the Kirwan map in the contact case such that surjectivity holds. In the setting of basic equivariant cohomology, our result states as follows.

\begin{theorem*}
Let $(M, \alpha)$ be a compact $K$-contact manifold and $\xi$ its Reeb vector field. Let $G$ be a torus that acts on $M$, preserving $\alpha$. Denote with $\Psi \colon M \rightarrow \g^*$ the contact moment map and suppose that $0$ is a regular value of $\Psi$. Then the inclusion $\Psi^{-1}(0)\subset M$ induces an epimorphism in equivariant basic cohomology
\[ H^*_G(M,\CF) \longrightarrow H^*_G(\Psi^{-1}(0), \CF).\]
\end{theorem*}

Kirwan's original proof makes use of the minimal degeneracy of the norm square of the symplectic moment map, a property that is weaker than the Morse-Bott property and which was established in \cite[Chapter 4]{kirwan1984cohomology}. The question of minimal degeneracy of the norm square of the contact moment map is still unanswered. Furthermore, Kirwan makes use of the topological definition of equivariant cohomology of a $G$-manifold $N$ as ordinary cohomology of the space $M\times_G EG$, where $EG$ denotes the total space in the classifying bundle of $G$. This tool is not available in the basic setting. Hence, Kirwan's approach does not naturally extend to the basic setting on $K$-contact manifolds. Instead, we obtain the epimorphism as a sequence of surjective maps. Goldin introduced the reduction in stages strategy in \cite{goldin2002effective}. She considers a splitting $S^1\times S^1 \times ... \times S^1$ of a subtorus $K \subset G$. By successively taking $S^1$-quotients, considering the residual 
action of the quotient group on the quotient and applying a surjectivity result for the $S^1$-case, 
she obtains a sequence of surjections
\[H_G(N) \to  H_{G/S^1}(N//S^1)\to H_{G/(S^1\times S^1)}((N//S^1)//S^1)\to \cdots \to H_{G/K}(N//K).\]
However, the quotient $N//S^1$ is in general an orbifold, not a manifold. Goldin's proof was made rigorous by Baird-Lin in \cite{baird2010topology}. Instead of considering a sequence of quotients, they rather consider a sequence of restrictions, retaining the action of the whole group. This idea was formulated by Ginzburg-Guillemin-Karshon in \cite[Section G.2.2]{ggk2002moment} for so-called non-degenerate abstract moment maps. Our approach is based on the proof of \cite[Theorem G.13]{ggk2002moment} and a corrected version thereof in \cite[Proposition~3.12,~Appendix~B]{baird2010topology}. The contact moment map, however, is in general not a non-degenerate abstract moment map (see Remark \ref{rem not non degenerate}), and \cite[Proposition~3.12]{baird2010topology} additionally requires a $G$-invariant almost complex structure. Hence, while providing an alternative proof of Kirwan surjectivity on symplectic manifolds, it does not hold in our case.

This paper is structured as follows. 
In Section \ref{sec contact mfd}, we recall fundamentals of $K$-contact geometry and consider torus actions on compact $K$-contact manifolds that leave the contact form invariant and allow for 0 to be a regular value of the contact moment map $\Psi$. 
We show that a basis $(X_s)$ of the Lie algebra of the torus can be chosen in such a way that certain axioms are fulfilled. 
We derive the Morse-Bott property of the functions $\Psi^{X_{s+1}}|_{Y_s}$, where $Y_s = (\Psi^{X_{1}}, ..., \Psi^{X_{s}})^{-1}(0)$. 
We begin Section \ref{sec eq basic coh} by briefly recalling the concept of equivariant cohomology and then proceed to discuss equivariant basic cohomology of $K$-contact manifolds and properties thereof. 
The surjectivity result is stated and proved in Section \ref{sec basic kirwan surj}. 
In Section \ref{sec ex}, we present examples and establish that in the case that the Reeb vector field induces a free $S^1$-action, our result reproduces the Kirwan surjectivity for the 
$S^1$-quotient. 
A Tolman-Weitsman type description of the kernel of the basic Kirwan map for $S^1$-actions is derived in Section \ref{sec kernel}. We conclude this paper by proving the equivariant formality of said torus actions on $K$-contact manifolds.

\section{\texorpdfstring{$K$}{K}-Contact Manifolds}\label{sec contact mfd}
\subsection{Preliminaries}
This section serves the purpose to establish terminology and notation and to recall basic facts on $K$-contact manifolds. 
We work with the following notion of contact manifolds.

\begin{definition}
	A \emph{contact manifold} is a pair $(M, \alpha)$, where $M$ is a manifold of dimension $2n+1$, and $\alpha \in \Omega^1(M)$ is a \emph{contact form}, i.e., $\alpha \wedge (d\alpha)^n$ is nowhere zero.
\end{definition}

Recall that there is a unique vector field $\xi \in \X(M)$, the \emph{Reeb vector field}, defined by $\alpha(\xi)=1$, $d\alpha(\xi, -) = 0$, so that we have a splitting $TM \cong \ker\alpha \oplus \R \xi$. 

The flow of $\xi$ is denoted by $\psi_t$ and the 1-dimensional foliation it induces by $\mathcal{F}$.

If $\xi$ induces a free $S^1$-action, $M/\{\psi_t\}$ is a manifold and $d\alpha$ descends to a symplectic form on $M/\{\psi_t\}$ (Boothby-Wang fibration). This, however, is usually not the case.

\begin{definition}
	Let $(M, \alpha)$ be a contact manifold. A Riemannian metric $g$ on $M$ is called \emph{contact metric} if 
	\begin{enumerate}[(i)]
	\item $ \ker \alpha \perp_g \ker d\alpha $
	\item $g|_{\ker d\alpha} = \alpha \otimes \alpha$ 
	\item $g|_{\ker \alpha}$ is compatible with the symplectic form $d\alpha$, i.e., there exists a $(1,1)$-tensor field $J$ on $\Gamma(\ker \alpha)$ such that $g=d\alpha(J\cdot, \cdot)$ and $J^2=-\operatorname{Id}$.
\end{enumerate} 
	$(M, \alpha)$ is called \emph{$K$-contact} if there exists a contact metric $g$ on $M$ with $L_{\xi} g = 0$, i.e., such that $\xi$ is Killing.
\end{definition}

\begin{example}\label{ex weighted sphere}
	For $n \geq 1$ and $w \in \R^{n+1}, w_j>0$, consider the sphere 
	\[S^{2n+1}= \left\lbrace z=(z_0, ..., z_n) \in \C^{n+1} \mid \sum\nolimits_{j=0}^n |z_j|^2 = 1 \right\rbrace \subset \C^{n+1},\]
	endowed with the following contact form $\alpha$ and corresponding Reeb vector field $\xi$
	\[\alpha_w =\frac{\tfrac{i}{2} \left( \sum_{j=0}^n z_j d\bar z_j - \bar z_j d z_j \right)}{\sum_{j=0}^n w_j|z_j|^2} , \quad \xi_w = i\left( \sum_{j=0}^{n} w_j(z_j \tfrac{\partial}{\partial z_j} - \bar z_j \tfrac{\partial}{\partial \bar z_j}) \right).\]
	$(S^{2n+1}, \alpha_w)$ is called a \emph{weighted Sasakian structure on $S^{2n+1}$}, cf. \cite[Example~7.1.12]{boyer2008sasakian}. In particular, $(S^{2n+1}, \alpha_w)$ is a $K$-contact manifold with respect to the metric induced by the embedding $S^{2n+1} \hookrightarrow \C^{n+1}$. For $w=(1, ..., 1)$, we obtain the standard contact form on the sphere. Notice that the underlying contact \emph{structure} $\ker \alpha_w$ is independent of the choice of weight $w$.
\end{example}

Throughout this paper, we consider a connected, compact $K$-contact manifold $(M,\alpha)$ with Reeb vector field $\xi$ and contact metric $g$, on which a torus $G$ acts in such a way that it preserves the contact form $\alpha$, i.e., $g^*\alpha = \alpha$ for every $g \in G$. 
We refer to, e.g., \cite[Section~2]{GNTequivariant} or \cite{blair1976contact} for preliminary considerations. 
Note that only from Lemma \ref{lem GxTinv metric} on we will assume the $G$-action to be isometric. We denote the Lie algebra of $G$ by $\g$ and the fundamental vector field induced by $X \in \g$ on $M$ by $X_M$, i.e.,
\[ X_M(x) = \left.\frac{d}{dt}\right|_{t=0} \exp(tX)\cdot x.\]

$(M,\alpha, g)$ admits a $(1,1)$-tensor $J$ such that we have the following identities  for all $X,Y \in \mathfrak{X}(M)$, where $\nabla$ denotes the Levi Civita connection of $g$ (see, e.g.,~\cite[pp.~25f,~p.~64]{blair1976contact})

\begin{align}
J\xi &= 0,  \quad \text{$J^2 = -\operatorname{Id} + \alpha \otimes \xi$},\label{eq J xi and J2}\\
 \alpha(X)&=g(\xi,X), \\
 g(X,JY) &=d\alpha(X,Y),\\
g(JX,JY)&=g(X,Y)-\alpha(X)\alpha(Y),\\
 (\nabla_X J)Y &= R(\xi,X)Y,\label{eq nabla J=R} \\ 
\nabla_X \xi &=-JX. \label{eq nabla xi}
\end{align}

We set \label{def M xi}
\begin{align*}
 M_\xi & := 
 \{ x \in M \mid \xi(x) \in T_x(G\cdot x)\},\\
 M_\varnothing & := 
 \{ x \in M \mid \xi(x) \notin T_x(G\cdot x)\}.
\end{align*}
 Recall that $T_x(G\cdot x) = \{X_M(x) \mid X \in \g\}$. For $x \in M_\xi$, choose any $X^x \in \g$ such that $X^x_M(x)=\xi(x)$. This $X^x$ is unique modulo $\g_x= \{X\in \g \mid X_M(x)=0\}$. We define the \emph{generalized isotropy algebra}\label{def gen isotropy} of $x \in M$ by
\[\widetilde{\mathfrak{g}}_x:=\left\lbrace X \in \g \mid X_M(x) \in \R \xi(x)\right\rbrace = \begin{cases}
                                                         \g_x \oplus \R X^x & x \in M_\xi\\
							 \g_x		    & x \in M_\varnothing
                                                        \end{cases}.\]
We denote by $T = \overline{\{\psi_t\}} \subset \mathit{Isom}(M,g)$ the torus acting on $M$ that is the closure of the Reeb flow in the isometry group of $(M,g)$. Note that $T$ is independent of the choice of contact metric. Since $G$ and $T$ commute (see Equation \eqref{eq fund vf foliate} below), we can then consider the action of the torus $H := G \times T$ on $M$. Since $\widetilde{\mathfrak{h}}_x = \mathfrak{h}_x \oplus \R \xi$, we have \[\widetilde{\mathfrak{g}}_x = \widetilde{\mathfrak{h}}_x \cap (\g \oplus \{0\}).\] 
Note that since $M$ is assumed to be compact, only finitely many different $\g_x$, $\mathfrak{h}_x$ and, hence, $ \widetilde \g_x$ occur.

\subsection{Basic Differential forms}

Let $\mathfrak{X}(\CF)$ denote the vector space of vector fields on $M$ that are tangent to the leaves of the foliation $\CF$, $\X(\CF) = C^\infty(M) \cdot \xi$. Differential forms whose contraction with and Lie derivative in the direction of an element of $\X(\CF)$ vanish are called \emph{basic}. Their subspace is denoted by 
\[\Omega(M, \CF):=\{\omega \in \Omega(M) \mid \mathcal{L}_X \omega = \iota_X \omega =0 \ \forall \ X \in \X(\CF) \}.\]
Cartan's formula directly yields that $\Omega(M, \CF)$ is differentially closed, i.e., for every $\omega \in \Omega(M, \CF)$, we have $d\omega \in \Omega(M, \CF)$. The cohomology of this subcomplex is called \emph{basic cohomology} of the foliated manifold $(M, \CF)$ and denoted by $H(M, \CF)$. For a more elaborate introduction to basic differential forms, also for more general foliations, the reader is referred to \cite{reinhart1959harmonic}.

By definition of the Reeb vector field, it is $0=\iota_\xi d\alpha$ and $1=\iota_\xi \alpha$, hence
\begin{equation}
	\mathcal{L}_\xi \alpha = d \iota_\xi \alpha + \iota_\xi d\alpha = 0. \label{eq lie xi}
\end{equation}
It follows that $\alpha$ is invariant under pullback by the Reeb flow 
\begin{equation}
	\psi_t^* \alpha \equiv \psi_0^* \alpha = \alpha. \label{eq flow alpha invariant}
\end{equation}
Furthermore, the uniqueness of the Reeb vector field implies that for every $p \in M$ and every $g \in G$, we have
\begin{equation}
	dg_p \xi(p)= \xi(gp). \label{eq dg reeb=reeb}
\end{equation}

This implies  that $[X_M, \xi]=0$ and, in particular, for every $f\cdot \xi \in \X(\CF)$ that
\begin{align}
	[X_M, f\cdot \xi]= X_M(f)\xi \in \X(\CF). \label{eq fund vf foliate}
\end{align}
 \begin{remark}
 	Equation \eqref{eq fund vf foliate} means that, given $X \in \g$, for every $Y \in \X(\CF)$, the commutator $[X_M,Y]$ is also an element of $\X(\CF)$; hence all fundamental vector fields are so called \emph{foliate vector fields} as defined by Molino (see \cite[Chapter~2.2]{molino1988riemannian}).
 \end{remark}
Recall the following definition (cf. \cite[Definition~3.1]{goertsches2010equivariant} or \cite[Chapter~2.2]{guillemin1999supersymmetry} for a formulation in the language of superalgebras).
\begin{definition}\label{def g-dga}
 Let $\mathfrak{k}$ be a finite dimensional Lie algebra and $A= \bigoplus A_k$ a $\mathbb{Z}$-graded algebra. $A$ is called a \emph{differential graded $\mathfrak{k}$-algebra} ($\mathfrak{k}$-dga) or $\mathfrak{k}^*$-algebra, if there exist derivations $d:A \rightarrow A$ of degree 1, $\iota_X:A \rightarrow A$ of degree $-1$, and $\mathcal{L}_XA \rightarrow A$ of degree 0 for every $X \in \mathfrak{k}$ such that $\iota_X$ and $\mathcal{L}_X$ are linear in $X$ and
 \begin{multicols}{2}

\begin{itemize}
 \item $d^2=0$,
\item $\iota_X^2 = 0$,
\item $[\mathcal{L}_X,\mathcal{L}_Y]=\mathcal{L}_{[X,Y]}$,
\item $[\mathcal{L}_X, \iota_Y] = \iota_{[X,Y]}$,
\item $\mathcal{L}_X = d\iota_X + \iota_Xd$.
\item[]
\end{itemize}
\end{multicols}

\end{definition}

\begin{lemma}\label{lem Omega(M,F) g-dga} With the usual differential $d$ inherited from $\Omega(M)$ and $\iota_X := \iota_{X_M}$, $\mathcal{L}_X := \mathcal{L}_{X_M}$, $\Omega(M, \CF)$ is a $\g$-dga. 
\end{lemma}
\begin{remark}
          In \cite{GNTequivariant}, Goertsches, Nozawa and T\"oben consider so-called \emph{transverse actions} of Lie algebras on foliated manifold, especially the action of $\mathfrak{t}/\R \xi$ on a $K$-contact manifold. In particular, they show that $\Omega(M,\CF)$ is a $ \mathfrak{t}/\R \xi$-dga, see \cite[Proposition~2,~(3.1)]{GNTequivariant}.
\end{remark}

\begin{proof}
The relations of Definition \ref{def g-dga} as well as the degrees of the derivations are inherited from those on $\Omega(M)$. Let $\omega \in \Omega(M, \CF)$, $X\in \g$, $Y\in \X(\CF)$. For the proof of the $\g$-dga structure, it remains to show that $\iota_X \omega$ and $\mathcal{L}_X \omega$ are again elements of $\Omega(M, \CF)$. Note that we have $\iota_{[Y,X_M]}\omega=\mathcal{L}_{[Y,X_M]}\omega=0$ by Equation \eqref{eq fund vf foliate}. Then
\begin{align*}
 \iota_Y \iota_X \omega & = -\iota_X \iota_Y \omega =0,\\
 \mathcal{L}_Y \iota_X \omega & = \iota_{[Y,X_M]}\omega + \iota_X \mathcal{L}_Y \omega =0,
\end{align*}
and, similarly, we obtain $ \iota_Y \mathcal{L}_X \omega = \mathcal{L}_Y \mathcal{L}_X \omega =0$.
\end{proof}
As a generalization of the example where a Lie group $K$ acting on a manifold induces the structure of a $\mathfrak{k}$-dga on the differential forms, consider the following definition (cf. \cite[Definition~2.3.1]{guillemin1999supersymmetry}). 
\begin{definition}
	Let $\mathfrak{k}$ denote the Lie algebra of an arbitrary Lie group $K$. A $K^*$-algebra is a $\mathfrak{k}$-dga $A$ together with a representation $\rho $ of $K$ as automorphisms of $A$, that is compatible with the derivations in the sense that for all $h \in K, X \in \mathfrak{k}$, it is
	\begin{multicols}{2}
	\allowdisplaybreaks
\begin{itemize}
 \item $\tfrac{d}{dt}\rho(\exp(tX))|_{t=0}=\mathcal{L}_X$,
 \item $\rho(h) \mathcal{L}_X\rho(h^{-1})=\mathcal{L}_{Ad_h X}$,
\item $\rho(h)\iota_X \rho(h^{-1}) = \iota_{Ad_h X}$,
\item $\rho(h) d \rho(h^{-1}) = d$.
\end{itemize}
\end{multicols}

\end{definition}
For a different formulation in the language of superalgebras, the reader is referred to \cite[Section~2.3]{guillemin1999supersymmetry}.
\begin{lemma}\label{lem Omega(M,F) G* alg}
	The torus action of $G$ on $M$ induces an action on $\Omega(M, \CF)$ by pullback, i.e., $g^*\omega \in \Omega(M, \CF)$ for every $g \in G$, $\omega \in \Omega(M, \CF)$, turning $\Omega(M, \CF)$ into a $G^*$-algebra.
\end{lemma}
 
\begin{proof}
 Let $g \in G$, $\omega \in \Omega(M, \CF)$, $Y\in \X(\CF)$. By Equation \eqref{eq dg reeb=reeb}, the vector field $dg(Y)$, defined by $dg(Y)(p)=dg_{g^{-1}p}(Y_{g^{-1}p})$, lies in $\X(\CF)$, and, since $\Omega(M, \CF)$ is differentially closed, we have $d\omega\in \Omega(M, \CF)$. Hence, we obtain
\begin{align*}
 \iota_Y g^* \omega &= g^* \iota_{dg Y} \omega=0,\\
\mathcal{L}_Y g^* \omega &= d\iota_Y g^* \omega + \iota_Y d g^* \omega =0 + \iota_Y g^* d \omega =0.
\end{align*}
The compatibility relations are  inherited from $\Omega(M)$.
\end{proof}

Note that since we are considering an abelian group, the fundamental vector fields satisfy $dg(X_M(p))=X_M(g\cdot p)$, for every $X \in \g$, $g \in G$, $p \in M$. Therefore, we obtain by an easy calculation, that, if $\omega \in \Omega(M,\CF)$ is $G$-invariant, then so are $\iota_X\omega$ and $\mathcal{L}_X\omega$ for every $X \in \g$.

\subsection{Contact Moment Map}\label{subsec cont mom map}
Recall that $(M,\alpha)$ is a connected, compact contact manifold with Reeb vector field $\xi$, on which a torus $G$ acts in such a way that it preserves the contact form $\alpha$. 
The \emph{contact moment map} on $M$ is the map $\Psi\colon M \rightarrow \g^*$, defined by $ \Psi^X := \Psi(\cdot)(X) := \iota_{X_M} \alpha$ for every $X \in \g$.
\begin{remark}\label{rem not non degenerate}
	$\Psi$ is an abstract moment map according to the definition in \cite{ggk2002moment}: $G$-invariance (i.e., $G$-equi\-va\-ri\-ance) stems from the $G$-invariance of $\alpha$ and for every closed subgroup $H \subset G$, the map $\Psi^H := \operatorname{pr}_{\mathfrak{h}^*}\circ \Psi \colon M \rightarrow \mathfrak{h}^*$ is zero on the points fixed by the $H$-action, $M^H$, thus it is in particular constant on the connected components of $M^H$. 
	In general, however, this map is not a \emph{non-degenerate} abstract moment map, again as defined in \cite{ggk2002moment}, since, in general, the inclusion $\{X_M = 0 \} \subset \mathit{Crit}(\Psi^X)$ is not an equality, see Equation \eqref{eq crit psi x} below. 
\end{remark}

By Cartan's formula, $d \iota_{X_M} \alpha = - \iota_{X_M} d \alpha$ for every $X \in \g$ since $\mathcal{L}_{X_M} \alpha = 0$. Furthermore, $ \ker d\alpha_x = \R \xi(x)$. This implies that the critical set of the $X$-component of $\Psi$ is given by
\begin{align}
 \mathit{Crit}(\Psi^X) & = \{ x \in M \mid X_M(x) \in \R \xi(x)\} = \{ x \in M \mid X \in \widetilde \g_x\}. \label{eq crit psi x}
\end{align}

Since $\alpha(\xi) \equiv 1$ and $(\Psi^X)^{-1}(0) = \{x \in M \mid \alpha_x(X_M(x)) = 0\}$, Equation \eqref{eq crit psi x} implies 
\begin{align}
 \mathit{Crit}(\Psi^X) \cap (\Psi^X)^{-1}(0) = \{x \in M \mid X_M (x)= 0 \}. \label{eq crit psiX cap psiX-1 0}
\end{align}
\begin{lemma}\label{lem MG=varnothing}
 Suppose that $0$ is a regular value of $\Psi$. Then $M$ has no $G$-fixed points, $M^G=\varnothing$.
\end{lemma}
\begin{proof}
 Since all fundamental vector fields vanish on $M^G$, the claim is a consequence of Equation \eqref{eq crit psiX cap psiX-1 0}.
\end{proof}
In analogy to the symplectic setting (cf., e.g., \cite[23.2.1]{cannas2001lectures}), we have the following.
\begin{lemma}\label{lem im dPsi is annihilator}
	Denote the annihilator of $\widetilde \g_x$ in $\g^*$ by $\widetilde \g_x^0$. The image of $d\Psi_x$ is exactly $\widetilde \g_x^0$.
\end{lemma}
\begin{proof}
	The image of the linear map $d\Psi_x$ is the annihilator of the kernel of its transpose. By Equation \eqref{eq crit psi x}, the kernel of $d\Psi_x^t$ is $\widetilde \g_x$.
\end{proof}

 The next proposition and the resulting Proposition \ref{prop Hess formula} are crucial to the proof of our main result. They are inspired by the idea of the proof of Theorem G.13 in \cite{ggk2002moment} and a corrected version thereof in \cite[Proposition~3.12,~Appendix~B]{baird2010topology}. However, \cite[Proposition~3.12]{baird2010topology} requires a non-degenerate abstract moment map and a $G$-invariant almost complex structure. Hence, while providing an alternative proof of Kirwan surjectivity on symplectic manifolds, it does not hold in our case.
 
\begin{proposition}\label{prop basis K-contact}
 Let $(M, \alpha)$ be a compact $K$-contact manifold and $\xi$ its Reeb vector field. Let $G$ be a torus that acts on $M$, preserving $\alpha$. Denote with $\Psi \colon M \rightarrow \g^*$ the contact moment map and suppose that $0$ is a regular value of $\Psi$. Then there exists a basis $(X_1, ..., X_r)$ of $\g$ such that for every $s = 1, ..., r$
\begin{enumerate}[(i)]
 \item $0 \in \R^s$ is a regular value of $f_s := (\Psi^{X_1}, ..., \Psi^{X_s}) \colon M \rightarrow \R^s$,
 \item $\{x \in M \mid (X_s)_M(x) = 0\} = \varnothing$,
 \item For all $\g_x$ of dimension at most $ r-s$, the following holds: \vspace{-3pt}\[\g_x \cap \bigoplus_{j=1}^{s}\R X_j  = \{0\}, \]
 \item For all $\widetilde{\g}_x $ of dimension at most $ r-s$, the following holds: \vspace{-3pt}
\[\widetilde{\g}_x \cap \bigoplus_{j=1}^{s}\R X_j  = \{0\},\]
 \item  The critical points $C_{s}$ of $f_{s}$ are  \begin{align*} C_{s}&=\{x \in M \mid \widetilde{\mathfrak{g}}_x \cap \oplus_{j=1}^{s} \ \R X_j\neq \{0\} \}=\{x \in M \mid \dim\widetilde{ \mathfrak{g}}_x >r-s \}.\\
  \intertext{In particular, with $C_0:= \varnothing$,}
  C_{s}&=C_{s-1} \ \dot \cup \ \{x \in M \mid \dim \widetilde{\mathfrak{g}}_x = r-s+1 \}.\end{align*}
 \end{enumerate}
\end{proposition}

\begin{remark}
	We remark that a basis with properties \textit{(i)-(iii)} of Proposition \ref{prop basis K-contact} exists on a contact manifold that is not necessarily $K$-contact, the proof is similar.
\end{remark}

\begin{remark}
 Note that, together with Equation \eqref{eq crit psiX cap psiX-1 0}, \textit{(ii)} implies that \linebreak $\mathit{Crit}(\Psi^{X_s}) \cap (\Psi^{X_s})^{-1}(0)$ is empty.
\end{remark}

\begin{proof}[Proof of Proposition \ref{prop basis K-contact}]
 Recall that there are only finitely many $\g_x$ and $\widetilde \g_x$ and that $\g$ does not occur as isotropy algebra (by Lemma \ref{lem MG=varnothing}). Set $\mathfrak{k}= \cup \g_x \cup \cup_{\widetilde{\mathfrak{g}}_x\neq \g}\widetilde{\mathfrak{g}}_x$ and denote its complement by $\mathfrak{a}_0 = \g \setminus \mathfrak{k}$; as complement of finitely many proper subspaces, $\mathfrak{a}_0$ is open and dense. 
 With Equation \eqref{eq crit psiX cap psiX-1 0}, it follows that \textit{(i)-(v)} hold for $s=1$ with an arbitrary $X_1 \in \mathfrak{a}_0$.

Now, let us suppose we already found $X_1, ..., X_{s_0}$ such that \textit{(i)} - \textit{(v)} hold for $s=1,..., s_0$; we will construct $X_{s_0+1}$. Set $W_{s_0}= \oplus_{j=1}^{s_0} \R X_j$. 
The following set is open and dense in $\g$:
\begin{align*}
	\mathfrak{a}_{s_0} &:= \g \setminus \left( \bigcup_{\{x\in M \mid \dim \g_x <r- s_0\}}\hspace{-20pt}\left( \g_x \oplus W_{s_0}\right)\quad \cup  \bigcup_{\{x\in M \mid \dim \widetilde\g_x <r- s_0\}}\hspace{-20pt}\left( \widetilde \g_x \oplus W_{s_0}\right) \right)
\end{align*}
i.e., $\mathfrak{a}_{s_0}$ consists of those $X_{{s_0}+1}$ s.t. \textit{(iii)} and \textit{(iv)} hold for $s=s_0+1$.
 Any $X_{{s_0}+1} \in \mathfrak{a}_0 \cap \mathfrak{a}_{s_0} \neq \varnothing$ will then obviously satisfy \textit{(ii)-(iv)}. To show that the remaining properties are satisfied as well, we need
\begin{lemma}\label{lem MGi cap Ys =varnothing}
 Set $M^{\g_p}= \{x \in M \mid \g_p \subset \g_x\}$ and $Y_s := f_s^{-1}(0)$. For every $\g_p$ of dimension $r-s>0$
 , it holds that\vspace{-10pt}
\begin{align}
 M^{\g_p} \cap Y_s = \varnothing \label{eq MGi cap Ys = varnothing}.
\end{align}
\end{lemma}
\begin{proof}
Let $x \in M^{\g_p} \cap Y_s$. By \textit{(iii)}, $\g_p$ and $\oplus_{j=1}^s \R X_j$ span all of $\g$ since their intersection is zero. We have $\oplus_{j=1}^s \R X_j \subset \ker \Psi(x)$ by the definition of $Y_s$ and $\g_p\subset \ker \Psi(x)$ because $\Psi(M^{\g_p})$ lies in the annihilator of $\g_p$. Thus $\Psi(x)=0$. 
But Lemma \ref{lem im dPsi is annihilator} implies that $M^{\g_p}$ cannot contain a regular point of $\Psi$, hence, $0 \notin \Psi(M^{\g_p})$ since 0 is a regular value of $\Psi$.
\end{proof}
Let us return to the proof of Proposition \ref{prop basis K-contact}. 
We can view $f_{s_0+1}$ as the composition of $\Psi$ and the restriction from $\g$ to $W_{s_0+1}:= \oplus_{j=1}^{s_0+1} \R X_j$. By Lemma \ref{lem im dPsi is annihilator}, the image of $d\Psi_x$ is $\widetilde \g_x^0$. Composing with the restriction yields that $(df_{s_0+1})_x$ is surjective if and only if $\widetilde \g_x \cap W_{s_0+1}=\{0\}$. Thus, we have
\begin{align*}
C_{s_0+1}=   C_{s_0} &\dot{~\cup~}  \left\lbrace x \in M \mid \widetilde{\mathfrak{g}}_x \cap W_{s_0} = \{0\}, \ \widetilde{\mathfrak{g}}_x \cap W_{s_0+1} \neq \{0\} \right\rbrace. \numberthis \label{eq crit fs+1 tx}
\end{align*}
Since we chose $X_{s_0+1}\in \mathfrak{a}_0 \cap \mathfrak{a}_{s_0}$, we directly obtain the remaining statement of \textit{(v)} for $s=s_0+1$, in particular, with $M_\varnothing$ and $M_\xi$ from page \pageref{def M xi}:
  \begin{align}
  C_{{s_0}+1}
   =C_{{s_0}} & \dot{~\cup~}  \underbrace{\{x \in M_\varnothing \mid \dim \g_x = r-{s_0}
   \}   }_{=:A_1} \ \dot{\cup} \ \underbrace{\{x \in M_\xi \mid \dim \widetilde{\mathfrak{g}}_x =r-s
  \} }_{=:A_2} \! .\label{eq crit fs+1 g1}
  \end{align}
It remains to show that \textit{(i)} holds for $s=s_0+1$. By assumption, 0 is a regular value of $f_{s_0}$, thus $C_{{s_0}}\cap Y_{s_0} = \varnothing$. Lemma \ref{lem MGi cap Ys =varnothing} yields that $A_1\cap Y_{s_0} = \varnothing$. Now, consider an element $x \in A_2 \cap Y_{s_0}$. Then $\widetilde{\mathfrak{g}}_x \cap W_{{s_0}} = \{0\}$, $\widetilde{\mathfrak{g}}_x \cap W_{{s_0}+1} \neq \{0\}$.  It follows that $X_{s_0+1}\in \widetilde \g_x \oplus W_{s_0}$. For every $X\in W_{s_0}$, $\Psi^X(x)=0$. Suppose $\Psi^{X_{s_0+1}}(x)=0$. Then, by definition of $\Psi$ and since $\alpha(\xi)=1$, it would follow that $X_{s_0+1}\in  \g_x \oplus W_{s_0}$. But this contradicts $X_{s_0+1}\in \mathfrak{a}_0 \cap \mathfrak{a}_{s_0}$. We showed that $0 \notin \Psi^{X_{{s_0}+1}}(C_{{s_0}+1}\cap Y_{s_0})$, meaning that  \textit{(i)} is satisfied for $s=s_0+1$. 
Hence, we showed that with any choice of $X_{s_0+1} \in \mathfrak{a}_0 \cap \mathfrak{a}^{s_0}   \neq \varnothing$, \textit{(i)} - \textit{(v)} hold for $s=s_0+1$.
\end{proof}

Recall that we set $f_s := (\Psi^{X_1}, ..., \Psi^{X_s}) \colon M \rightarrow \R^s$ and $Y_s := f_s^{-1}(0)$. 
\begin{lemma}\label{lem g(tilde) cap Ws =0 for x in Ys - dim g(tilde)} 
	With $(X_s)$ as in Proposition \ref{prop basis K-contact}, we have for every $x \in Y_s$ 
	\begin{equation}
 \{0\} = \widetilde{\g}_x \cap \oplus_{j=1}^{s} \R X_j. \label {eq lemma}
\end{equation}
In particular, $\dim \widetilde \g_x \leq r-s$ and $\dim \g_x < r-s$.
\end{lemma}
\begin{proof}
	Equation \eqref{eq lemma} follows directly from Proposition \ref{prop basis K-contact}, by combining  \textit{(i)},  \textit{(iv)} and \textit{(v)}. It directly implies that  $\dim \widetilde \g_x \leq r-s$. Since $\dim \g_x \leq \dim \widetilde \g_x$, and $Y_s$ does not contain a point with isotropy of dimension $r-s$ by Lemma \ref{lem MGi cap Ys =varnothing}, it follows that $\dim \g_x < r-s$.
\end{proof}

A main aspect needed for the proof of our main Theorem will be the Morse-Bott property of the functions $\Psi^{X_{s+1}}|_{Y_s}$. As a first step, we now want to compute their critical sets $\mathit{Crit}(\Psi^{X_{s+1}}|_{Y_s})$. Recall that $T$ denotes the closure of the flow of the Reeb vector field $\xi$ in the isometry group of $(M,g)$, where $g$ is any contact metric, and that $T$ is independent of the choice of $g$.
\begin{lemma}\label{lem crit= min Gpsi (GT)-orbits}
	With $(X_s)$ as in Proposition \ref{prop basis K-contact}, $\mathit{Crit}(\Psi^{X_{s+1}}|_{Y_s}) $ is the union of all the minimal $G\times \{\psi_t\}$-orbits, i.e., of all $G\times \{\psi_t\}$-orbits of dimension $s+1$. They coincide with the minimal $G\times T$-orbits. These are exactly the points of $Y_s$ with generalized isotropy algebra of dimension $r-s$. In particular, $\mathit{Crit}(\Psi^{X_{1}})=\mathit{Crit}(\Psi) $ consists of all points with $\widetilde \g_x =\g$.
\end{lemma}

\begin{proof}
Set $(Y_s)_\xi:= Y_s \cap M_\xi$ and $(Y_s)_\varnothing:= Y_s \cap M_\varnothing$, with $M_\varnothing$ and $M_\xi$ from page \pageref{def M xi}. We first show that $\mathit{Crit}(\Psi^{X_{s+1}}|_{Y_s})= \bigcup_{\substack{ x \in (Y_s)_\xi \\ \dim G \cdot x = s+1}} G \cdot x $.
Let $x \in Y_s$. By \textit{(i)} of Proposition \ref{prop basis K-contact}, we know that 
$\operatorname{span} \{d\Psi^{X_{1}}_x, ..., d\Psi^{X_{s}}_x\}$ is $s$-dimensional and $T_xY_s=\ker (df_s)_x$. Since the annihilator of $T_xY_s$ in $T_x^*M$ is $s$-dimensional, it follows that $T_xY_s$ lies in the kernel of a 1-form if and only if that 1-form lies in the span of $\{ d\Psi^{X_{1}}_x, ..., d\Psi^{X_{s}}_x\}$. Therefore, we obtain 
\begin{align*}
 \mathit{Crit}(\Psi^{X_{s+1}}|_{Y_s}) &= \left\lbrace x \in Y_s \ \left| \ (d\Psi^{X_{s+1}})_x \in \operatorname{span}\{ d\Psi^{X_{1}}_x, ..., d\Psi^{X_{s}}_x\}\right\rbrace \right.
\end{align*}
Using additivity of $d\Psi^X$ in $X$ and applying Equation \eqref{eq crit psi x}, this equation becomes
\begin{equation}
	\mathit{Crit}(\Psi^{X_{s+1}}|_{Y_s}) = \left\lbrace x \in Y_s \mid X_{s+1} \in \widetilde{\g}_x \oplus W_s \right\rbrace , \numberthis \label{eq crit restricted to Ys}
\end{equation}
 where  $W_s=\oplus_{j=1}^s \R X_j$.

By Lemma \ref{lem g(tilde) cap Ws =0 for x in Ys - dim g(tilde)}, $\dim \g_x<r-s$ and $\dim \widetilde{\g}_x \leq r-s$ for every $x \in Y_s$. With \textit{(iv)} of Proposition \ref{prop basis K-contact}, the condition in Equation \eqref{eq crit restricted to Ys} can then only be satisfied for $x \in Y_s$ with $\dim \widetilde{\g}_x = r-s$, 
thus $x \in (Y_s)_\xi$. Since in that case, it is $ \g = \widetilde{\g}_x \oplus W_s$, we automatically obtain that $ X_{s+1} \in \widetilde{\g}_x \oplus W_s$. 
 Hence,
\begin{align*}
 \mathit{Crit}(\Psi^{X_{s+1}}|_{Y_s}) &= \left\lbrace x \in (Y_s)_\xi \ \left| \ \dim \widetilde{\g}_x = r-s\right\rbrace \right.
= \bigcup_{\substack{ x \in (Y_s)_\xi \\ \dim G \cdot x = s+1}} G \cdot x 
\end{align*}
	Let $x \in (Y_s)_\varnothing$. From Lemma \ref{lem g(tilde) cap Ws =0 for x in Ys - dim g(tilde)}, we have $\dim \g_x \leq r-s-1$. Hence, $\dim (G\times T)\cdot x\geq \dim (G\times \{\psi_t\})\cdot x > \dim G \cdot x \geq s+1$, so the $G\times T$- and $G\times \{\psi_t\}$-orbits through $x$ are not minimal. Now, let $x \in (Y_s)_\xi$ and suppose that $\dim G \cdot x = s+1$ is minimal. By definition of  $(Y_s)_\xi$, $ \{\psi_t\} \cdot x \subset G \cdot x$, thus $\dim (G\times \{\psi_t\}) \cdot x = s+1$ as well. $G \cdot x$ is closed, hence the same holds for $T$: $ T \cdot x \subset G \cdot x$ and $\dim (G\times T) \cdot x = s+1$ is minimal.
\end{proof}

\begin{lemma}\label{lem GxTinv metric}
 There exists a contact metric $g$ on $M$ such that all $G$-fun\-da\-men\-tal vector fields are Killing vector fields, i.e., such that $g$ is $G \times T$-invariant.
\end{lemma}
\begin{proof}
 Choose any $G\times T$-invariant and $d\alpha$-compatible metric $h$ on $\ker \alpha$, which has to exist since $G\times T$ is compact. Then $g:= h \ \oplus \ \alpha \otimes \alpha$ is a $G \times T$-invariant contact metric on $M$.
\end{proof}
Now, let $N \subset \mathit{Crit}(\Psi^{X_{s+1}}|_{Y_s})$ be a connected component of the critical set.
From now on, we will work with a metric according to Lemma \ref{lem GxTinv metric}, i.e., with an isometric $G\times T$-action.

\begin{lemma}\label{lem tot geod}
 $N$ is a totally geodesic closed submanifold of even codimension.
\end{lemma}
\begin{proof}
By Lemma \ref{lem crit= min Gpsi (GT)-orbits}, $N$ is a union of minimal dimensional $G \times T$-orbits. 
The isotropy group of a point in a tubular neighborhood of an orbit $(G \times T)\cdot p$ is a subgroup of $(G\times T)_p$. By minimality, every point of $N$ in that tubular neighborhood then has to have the same isotropy algebra, 
so $\{x \in N \mid (\g \times \mathfrak{t})_x = (\g \times \mathfrak{t})_p \}$ is open in $N$. Since $N$ is connected, it follows that the connected component of the isotropy remains the same along $N$, $(\g \times \mathfrak{t})_x =: (\g \times \mathfrak{t})_N$ for all $x \in N$. Since all fundamental vector fields are Killing, we can apply a result of Kobayashi \cite[Corollary~1]{kobayashi1958fixed}, which directly yields that $N$ is a totally geodesic closed submanifold of even codimension.
\end{proof}

We will denote the $g$-orthogonal normal bundle of $N$ in $Y_s$ by $\nu N$, $T_pY_s = T_pN \oplus_{\bot_g} \nu_p N$. We will now prove the Morse-Bott property of $\Psi^{X_{s+1}}|_{Y_s}$. 
\begin{proposition}\label{prop Hess formula}
 The Hessian $H$ of $\Psi^{X_{s+1}}|_{Y_s}$ along $N$ in normal directions 
 is given by
\begin{align*}
 H_p (v,w) & = 2g(w, \nabla_v (JY)) = 2g(w, J\nabla_v Y),
\end{align*}
where $p \in N$, $Y:=(X_{s+1})_{Y_s} - \alpha ((X_{s+1})_{Y_s})_p\xi$, and $g$ is a metric as in Lemma \ref{lem GxTinv metric}.

Furthermore, the vector $J\nabla_v Y$ is normal and non-zero for every normal vector $v\neq 0$ and $H$ is non-degenerate in normal directions. 

In particular, $\Psi^{X_{s+1}}|_{Y_s}$ is a Morse-Bott function.
\end{proposition}
\begin{proof}
 Let $p\in N$ and $v,w\in \nu_pN$ be arbitrary. In a sufficiently small neighborhood of $p$, extend $v$ and $w$ to local vector fields $V,W$ around $p$ such that $(\nabla V)(p)=(\nabla W)(p)=0$. To shorten notation, let $X:= (X_{s+1})_{Y_s}$. Note that since $[X,\xi]=0$ by Equation \eqref{eq fund vf foliate}, we have $\nabla_X \xi = \nabla_\xi X$. The first computation in \cite[Section~2]{rukimbira1999} is equally applicable in our case since $X$ is a Killing vector field, hence we obtain at $p$, applying Equations \eqref{eq nabla J=R} and \eqref{eq nabla xi},
 \allowdisplaybreaks
\begin{align*}
 H_p(v,w) & = \left( V(W(\alpha(X)) \right)(p)=  \left( V(W(g(\xi,X))) \right)(p)\\
&= \left( V(g(\nabla_W \xi,X)+g( \xi,\nabla_WX)) \right)(p)\\
&= \left( V(g(-JW ,X)-g( \nabla_\xi X,W))\right) (p)\\
&= \left( -g(\nabla_V JW ,X)-g( JW ,\nabla_VX) +V(g( J X,W))\right) (p)\\
&= \left( -g(\nabla_V JW ,X)-g( JW ,\nabla_VX) +g(\nabla_V J X,W)+g( J X,\nabla_VW)\right) (p)\\
&= \left( -g((\nabla_V J)W ,X)-g( J(\nabla_VW) ,X)+g( W ,J\nabla_VX) +g((\nabla_V J) X,W)\right.\\
 &\quad + \left. g(J(\nabla_VX),W)+g( J X,\nabla_VW)\right) (p)\\
&= \left( -g(R(\xi,V)W,X)+2g( W ,J\nabla_VX) +g(R(\xi,V)X,W)\right) (p)\\
&= \left( 2g(R(\xi,V)X,W)+2g( W ,J\nabla_VX)\right) (p).\numberthis \label{eq hess 1}
\end{align*}
Combining Lemma \ref{lem tot geod}, Equation \eqref{eq nabla xi}, and the fact that $\xi(x) \in T_xN$ for all $x \in N$, we obtain that $Jz=-\nabla_z\xi  \in TN$ for all $z \in TN$, hence 
\[ J: T_pN \rightarrow T_pN, \quad J: \nu_pN \rightarrow \nu_pN.\]
Set $a:=  \alpha ((X_{s+1})_{Y_s})_p$ and decompose $X$ as $X=a\xi + Y$. It is
\[(\nabla_VX)(p)=(a\nabla_V \xi+\nabla_V Y)(p) =-a Jv + (\nabla_V Y)_p.\]
Using the tensor properties of the curvature tensor $R$ and that $R(\xi,V)\xi=-V$ (see \cite[p.~65]{blair1976contact}), we can then continue Equation \eqref{eq hess 1} as follows:
\begin{align*}
 H_p(v,w) & =\left(2ag(R(\xi,V)\xi,W)\!+\!2g(R(\xi,V)Y,W)\!+\!2g( W ,J(-a JV \!+\! \nabla_V Y))\right) (p)\\
 &= -2ag(v,w)+ 2g(R(\xi,V)Y,W)(p) +2ag(v,w) +2g(W, J\nabla_V Y)(p)\\
 &= 2g(R(\xi,V)Y,W)(p)+ 2g(W, J\nabla_V Y)(p)\\
 &= 2g(W, (\nabla_V J)Y + J\nabla_V Y)(p)\\
 &= 2g(W, \nabla_V (JY))(p). \numberthis \label{eq hess 2}
\end{align*}

 It remains to show that $ \nabla_V (JY)(p)=(R(\xi,V)Y+J\nabla_V Y)(p)$ equals $J\nabla_V Y(p)$, is non-zero, and lies in $\nu_pN$. Let $\eta$ be an arbitrary vector field in a neighborhood of $p$ that is tangent to $N$ at $p$. By Lemma \ref{lem tot geod}, $\nabla_\eta X (p) \in T_pN $. Since $X$ is Killing and $v \in T_pN^{\perp_g}$, we then have $g(\eta, \nabla_V X)_p = -g(\nabla_\eta X, V)_p = 0$.
Thus, $\nabla_V X (p) \in \nu_pN$ and, hence, $J\nabla_V X (p)\in \nu_pN$. With Equations \eqref{eq nabla xi}, \eqref{eq J xi and J2} and $\alpha(V)_p=g(\xi, V)_p=0$, we obtain
\[ g(\eta, J\nabla_V Y)_p =g(\eta, J\nabla_V X)_p - g(\eta, J\nabla_V (a\xi))_p = -g(\eta, aV)_p = 0,\]
hence $(J\nabla_V Y)(p) \in (T_pN)^{\bot_g}=\nu_pN$. Analogously, we obtain $(\nabla_V Y)(p) \in \nu_pN$.\\
Recall that $\mathcal{L}_X \alpha = 0$. Since $N$ is critical, we obtain on $N$
\[0=-d\iota_X \alpha = \iota_X d \alpha = a\iota_\xi d \alpha + \iota_Y d \alpha = \iota_Y d\alpha.\]
But $Y|_N \in \Gamma(\ker \alpha)$ since $\alpha(X)|_N \equiv a$, and $d \alpha$ is non-degenerate on $\ker \alpha$; therefore, $Y = 0$ on $N$. Hence, $ \nabla_V (JY)(p)=(R(\xi,V)Y+J\nabla_V Y)(p)= (J\nabla_V Y)(p)$.

We now follow the line of argumentation of Rukimbira in \cite[Proof of Lemma 1]{rukimbira1995topology} to show that $\nabla_v Y$ does not vanish on $N$. Note that $Y$ is a Killing vector field since $X$ and $\xi$ are. Let $\gamma$ be the geodesic through $\gamma(0)=p$ with tangent vector $\dot \gamma(0)=v$. Suppose $(\nabla_v  Y)(p)=0$. Then the Jacobi field $ Y \circ \gamma$ satisfies $ Y \circ \gamma (0)=0$ and $\tfrac{\nabla}{dt}(  Y \circ \gamma) (0)=0$, thus $ Y$ vanishes along all of $\gamma$. But this means that along $\gamma$, $X = a\xi$; by Equation \eqref{eq crit psi x}, $\gamma$ hence consists of critical points of $\Psi^X|_{Y_s}$. Thus, $\gamma$ lies in $N$ and $v$ has to be tangent to $N$. This, however, contradicts $v \in \nu_pN$. We conclude that $\nabla_V Y(p)$ is non-zero. Since $\nabla_V Y(p)$ is normal and, hence, lies in $\ker \alpha$, it follows that $J(\nabla_V Y)(p)$ is non-zero. Then we have for every non-zero normal vector $v \in \nu_pN$:
\[H_p(v, J(\nabla_v  Y)) = 2g(J(\nabla_v  Y),J(\nabla_v Y))= 2g(\nabla_v Y,\nabla_v  Y)\neq 0.
\quad \qedhere\]
\end{proof}
\begin{remark}
	$J$ is skew-symmetric w.r.t. $H$: For $v,w \in \nu_pN$, we have
	\begin{align*}
		\tfrac{1}{2}H_p(w,Jv)&= g(w,J\nabla_{Jv}Y)=-g(Jw, \nabla_{Jv}X + aJ^2v)= g(\nabla_{Jw}X,Jv) +ag(Jw,v) \\
		&= g(\nabla_{Jw}Y,Jv)+g(-aJ^2w,Jv)-ag(w, Jv)\\
		&= -g(J\nabla_{Jw} Y,v)+ag(w,Jv)-ag(w,Jv) = -\tfrac{1}{2}H_p(Jw,v).
	\end{align*}
	In particular, $J$ preserves the positive and negative normal bundle, $J: \nu^\pm N \rightarrow \nu^\pm N$.
\end{remark}

\section{Equivariant basic cohomology}\label{sec eq basic coh}
\subsection{Equivariant Cohomology of a \texorpdfstring{$\mathfrak{k}$}{k}-dga}
We will briefly review the concept of equivariant cohomology. For a more elaborate introduction, we refer to \cite{guillemin1999supersymmetry}, presenting the material from Cartan (cf.~\cite{cartan1950}) in a modern reference; see also \cite[Section~4]{GNTequivariant} or \cite[Section~3.2]{goertsches2010equivariant}.

The \emph{Cartan complex} of a $\mathfrak{k}$-dga $A$ is defined as 
\[C_\mathfrak{k}(A):= (S(\mathfrak{k}^*)\otimes A)^\mathfrak{k},\]
where $S(\mathfrak{k}^*)$ denotes the symmetric algebra of $\mathfrak{k}^*$ and the superscript denotes the subspace of $\mathfrak{k}$-in\-va\-ri\-ant elements, i.e.,those $\omega \in S(\mathfrak{k}^*)\otimes A$ for which $L_X \omega = 0$ for every $X \in \mathfrak{k}$. When regarding an element $\omega \in C_\mathfrak{k}(A) $ as a $\mathfrak{k}$-equivariant polynomial map $\mathfrak{k} \rightarrow A$, i.e., $\omega([X,Y])=L_X(\omega(Y))$ for every $X, Y \in \g$, the differential $d_\mathfrak{k}$ of $C_\mathfrak{k}(A)$ is given by
\[(d_\mathfrak{k} \omega)(X) := d(\omega(X)) - \iota_X (\omega(X)).\] 
If $\{X_i\}_{i=1}^r$ is a basis of $\mathfrak{k}$ with dual basis $\{u_i\}_{i=1}^r$, the differential can be written as
\[d_\mathfrak{k}(\omega) = d(\omega)- \sum_{i=1}^r \iota_{X_i}(\omega)u_i. \]
$C_\mathfrak{k}(A)$ can be endowed with the grading $\deg(f \otimes \eta)= 2\deg(f) + \deg(\eta)$. Then $d_\mathfrak{k}$ raises the degree by $1$.
The \emph{equivariant cohomology} of $A$ (in the Cartan model) is then defined by
\[H^*_\mathfrak{k}(A):= H^*(C_\mathfrak{k}(A), d_\mathfrak{k}).\]
We remark that there are different conventions in the literature concerning the sign in the definition of the differential.

\begin{example}
	If a compact Lie group $K$ acts on a manifold $N$, this action induces a $\mathfrak{k}$-dga structure on the algebra of differential forms $\Omega(N)$. This enables us to apply the general construction of the equivariant cohomology of a $\mathfrak{k}$-dga and we obtain the equivariant cohomology of the $K$-action as
	\[H^*_K(N) = H^*_\mathfrak{k}(\Omega(N)).\]
\end{example}

For the following definition compare 
\cite[Definition~2.3.4]{guillemin1999supersymmetry}.
\begin{definition}
	A $\mathfrak{k}$-dga $A$ is called \emph{free}, if, given a basis $X_i$ of $\mathfrak{k}$, there are $\theta_i \in A_1$ (called \emph{connection elements}) such that $\iota_{X_j} (\theta_i) =\delta_{ij}$. If, in addition, the $\theta_i$ can be chosen such that their span in $A_1$ is $\mathfrak{k}$-invariant, then $A$ is said to be \emph{of type (C)}.
\end{definition}
\begin{lemma}\label{lemma group action of type c}
	A free $\mathfrak{k}$-dga $A$ is automatically of type (C) if the action of $\mathfrak{k}$ on $A$ is induced by an action of a compact Lie group.
\end{lemma}
\begin{proof}
	\cite[Section~2.3.4]{guillemin1999supersymmetry}.
\end{proof}

\begin{definition}
	Let $A$ be a $\mathfrak{k}$-dga. The differentially closed set $A_{\mathrm{bas}} := \{ \omega \in A \mid \iota_X \omega = 0 = L_X \omega \text{ for every } X \in \mathfrak{k} \}$ is called the \emph{basic subcomplex} of $A$.
\end{definition}

\begin{proposition}\label{prop H G = H(bas) if of type C}
	If $A$ is a $\mathfrak{k}$-dga of type (C), then $H^*_{\mathfrak{k}}(A) = H^*(A_{\mathrm{bas}\, \mathfrak{k}})$.
\end{proposition}
\begin{proof}
	\cite[Section~5.1]{guillemin1999supersymmetry}.
\end{proof}

A proof of the following proposition can be found in \cite[Section~4.6]{guillemin1999supersymmetry} or \cite[Proposition~3.9]{goertsches2010equivariant}.

\begin{proposition}\label{prop commuting actions}
Let $A$ be an $(\mathfrak{h} \times \mathfrak{k})$-dga with $A_k = 0$ for $k < 0$, which
is of type (C) as an $\mathfrak{h}$-dga. If either $A^\mathfrak{k} = A$ or $\mathfrak{k}$ is the Lie algebra of the compact connected Lie group $K$ and the $\mathfrak{k}$-dga structure on $A$ stems from a $K^*$-algebra structure, then
\[H^*_{\mathfrak{h} \times \mathfrak{k}}(A) = H^*_\mathfrak{k} (A_{\mathrm{bas}\, \mathfrak{h}})\]
as $S(\mathfrak{k}^*)$-algebras. The isomorphism is induced by the natural inclusion of complexes
\[\left( \left( S(\mathfrak{k}^*) \otimes A_{\mathrm{bas}\, \mathfrak{h}}\right)^{\mathfrak{k}}, d_\mathfrak{k}\right) \hookrightarrow  \left( \left( S(\mathfrak{k}^*)\otimes S(\mathfrak{h}^*)\otimes A\right)^{\mathfrak{k}\times \mathfrak{h}}, d_
{\mathfrak{k}\times \mathfrak{h}}\right).\]
\end{proposition}

\subsection{Equivariant basic Cohomology}\label{subsec eq bas coh}
Recall that we consider a connected, compact $K$-contact manifold $(M,\alpha)$ with Reeb vector field $\xi$, on which a torus $G$ acts in such a way that it preserves the contact form $\alpha$, i.e., $g^*\alpha = \alpha$ for every $g \in G$. We denoted the Reeb flow by $\{\psi_t\}$.
	We can not only consider the $G$-action on $M$, obtaining the equivariant cohomology of the $G$-action as $H^*_G(M) = H^*_\g(\Omega(M))$,
	but we can also consider the $G\times \{\psi_t\}$-action on $M$ which induces a $\g \times \R \xi$-dga structure on $\Omega(M)$. This yields $H^*_{\g \times \R \xi}(M)$. Furthermore, by Lemma \ref{lem Omega(M,F) G* alg}, $\Omega(M,\CF)$ is a $G^*$-algebra (and especially a $\g$-dga). Set $C_G(M, \CF)= C_\g(\Omega(M, \CF))$. We obtain the \emph{equivariant basic cohomology} of the $G$-action on $(M,\alpha)$ as
\[H_G(M,\CF)=H(C_G(M, \CF), d_\g).\]

Analogously, we can define equivariant basic cohomology for any open or closed $G \times \{\psi_t\}$-invariant submanifold of $M$
or for any foliated manifold $(N,\mathcal{E})$, acted on by a torus $H$ s.t. $\Omega(N,\mathcal{E})$ is an $H^*$-algebra.

Since the $G$-invariant contact form serves as connection element, Proposition \ref{prop commuting actions} directly gives
\begin{proposition}\label{prop H G (M,F)=H Gx xi}
	$H_G(M,\CF)= H_{\mathfrak{g}\times \R \xi}(M)$ as $S(\g^*)$-algebras.
\end{proposition}
\begin{example}\label{ex xi free S1 action}
	Suppose that $\xi$ induces a free $S^1$-action. In this case, $\{\psi_t\}=S^1=T$ and $\pi \colon M \rightarrow M/\{\psi_t\} =: B$ is a $G$-equivariant principal $S^1$-bundle. The pullback gives an isomorphism $\pi^* : \Omega(B) \rightarrow \Omega(M, \CF)$ and we have $H_G(M, \CF) = H_G(B)$ (compare~\cite[Example~3.14]{goertsches2010equivariant}). 
\end{example}

\begin{lemma}\label{lem only 1 g tilde}
	Assume $G$ acts on  
	a $G\times T$-invariant submanifold $U \subset M$ with only one $\widetilde \g_x=\widetilde \g_U$; then $H^*_G(U, \CF)=S(\widetilde \g_U^*)\otimes H_{\mathfrak{k}}^*(U, \CF) = S(\widetilde \g_U^*)\otimes H^*(\Omega(U, \CF)_{ \mathrm{bas } \mathfrak{k}})$, where $\mathfrak{k}$ denotes a complement of $\widetilde \g_U $ in $\g$.
\end{lemma}
\begin{proof}
	Since $\widetilde \g_U$ acts trivially on $\Omega (U, \CF)$, we can write the Cartan complex as  $C_G(U, \CF) = S(\widetilde \g_U^*)\otimes S(\mathfrak{k}^*) \otimes \Omega (U, \CF)^{\mathfrak{k}}$ and $d_G = 1 \otimes d_{\mathfrak{k}}$, hence $H^*_G(U, \CF)=S(\widetilde \g_U^*)\otimes H_{\mathfrak{k}}^*(U, \CF)$. But $\mathfrak{k}$ acts freely and in transversal direction on $U$, so $\Omega(U,\CF)$ is a $\mathfrak{k}$-dga of type (C) and $H_{\mathfrak{k}}^*(U, \CF)=  H^*(\Omega(U, \CF)_{ \mathrm{bas } \mathfrak{k}})$ by Proposition \ref{prop H G = H(bas) if of type C}.
\end{proof}

\begin{proposition}[Mayer-Vietoris sequence]\label{prop MVS}
	Let $A\subset M$ be a 
	$G \times T$-invariant submanifold of $M$ and let $U,V\subset A$ be open $G\times T$-invariant subsets such that $U \cup V=A$. Denote the inclusions by $i_U: U \to A$, $i_V:V\to A$, $j_U:U\cap V \to U$, $j_V: U\cap V \to V$. Then there is a long exact sequence
	\[...\! \to H^*_G(A,\CF)\stackrel{i_U^*\oplus i_V^*}{\to} H^*_G(U,\CF)\oplus H_G^*(V,\CF) \stackrel{j_U^* - j_V^*}{\to} H_G^*(U\cap V,\CF) \to H^{*+1}_G(A,\CF)\to \!...\]
\end{proposition}
\begin{proof}
	Since $U$ and $V$ are $G\times T$-invariant and $G\times T$ is compact, we can find a $G\times T$-invariant partition of unity subordinate to the open cover ${U, V}$ of $A$ (\cite[Corollary~B.33]{ggk2002moment}). Then it can be seen as done in \cite[Proposition~2.3]{bott2013differential} for ordinary differential forms that we have a short exact sequence
	\[0 \to C^*_G(A,\CF)\stackrel{i_U^*\oplus i_V^*}{\to} C^*_G(U,\CF)\oplus C_G^*(V,\CF) \stackrel{j_U^* - j_V^*}{\to} C_G^*(U\cap V,\CF) \to 0. \]
	Thus, we obtain a long exact sequence in equivariant basic cohomology.
\end{proof}

We will not only work with $G \times T$-invariant submanifolds of $(M, \CF)$, but also with the (positive/negative) normal bundles of closed invariant submanifolds with lifted $G\times T$-action. For this reason, we now consider the more general case of a foliated manifold $(N,\mathcal{E})$ that is endowed with a $G\times T$-action such that the fundamental vector field of $\xi \in \mathfrak{t}$ is nowhere vanishing and induces $\mathcal{E}$.
\begin{definition}
	A subset $A \subset N$ is called \emph{$\mathcal{E}$-saturated} if for every $x \in A$, $A$ contains the whole leaf of $\mathcal{E}$ that runs through $x$.
\end{definition}
To prove our main result, we also need relative and compactly supported equivariant basic cohomology. Our assumption on $(N,\mathcal{E})$ means in particular that for any closed $G$-invariant, $\mathcal{E}$-saturated submanifold of $N$, we can find arbitrarily small $G$-invariant, $\mathcal{E}$-saturated tubular neighborhoods. 
\begin{definition}
We denote the subcomplex of compactly supported equivariant basic differential forms by $C_{G,c}(N,\mathcal{E})$, and its cohomology by $H_{G,c}(N,\mathcal{E})=H(C_{G,c}(N,\mathcal{E}), d_\mathfrak{g})$.\end{definition}
\begin{proposition}\label{prop homotopic same cohom}
	Let $A, B \subset N$ be two $G$-invariant, $\mathcal{E}$-saturated submanifolds that are $G \times \{\psi_t\}$-equi\-va\-ri\-antly homotopy equivalent. Then the homotopy inverse maps between $A$ and $B$ induce an isomorphism $H_G(A,\mathcal{E})=H_G(B,\mathcal{E})$.
	If, in addition, the homotopy is proper, the same holds for cohomology with compact support.
\end{proposition}
\begin{proof}
	The proposition is proven analogously to the corresponding statement in ordinary (equivariant) cohomology by constructing a chain homotopy, see, e.g., \cite[\S 4; Cor. 4.1.2]{bott2013differential} and also \cite[Section~2.3.3 and Proposition~ 2.4.1]{guillemin1999supersymmetry} and the proof of Proposition \ref{prop homotopy equiv relative} below.
\end{proof}
\begin{definition}\label{def rel eq coh}
Let $A\subset N$ be any $G$-invariant, $\mathcal{E}$-saturated submanifold. We then consider the complex $C_G(N,A,\mathcal{E}):= C_G(N,\mathcal{E})\oplus C_G(A,\mathcal{E})$ with the grading
$C^k_G(N,A,\mathcal{E}):= C^k_G(N,\mathcal{E})\oplus C^{k-1}_G(A,\mathcal{E})$ and differential $D(\alpha, \beta):= (d_G \alpha, \alpha|_A - d_G\beta)$. The cohomology of this complex is the \emph{relative equivariant basic cohomology} of $(N,A)$ and denoted by $H^*_G(N,A,\mathcal{E})$.
\end{definition}

This definition is based on the definition of ordinary relative de Rham cohomology in \cite[pp.~78-79]{bott2013differential} and an equivariant version thereof in \cite[Section~3.1]{paradanverne2007relative}. We remark that \cite[Section~3.1]{paradanverne2007relative} works analogously for closed submanifolds. Note that a $G\times \{\psi_t\}$-equivariant map of pairs $f: (N,A) \to (\tilde N, \tilde A)$, $f(A) \subset \tilde A$, induces a map $f^*:C_G(\tilde N, \tilde A,\tilde{ \mathcal{E}}) \to C_G(N,A,\mathcal{E}) $, $f^*(\alpha, \beta)=(f^*\alpha, f|_A^* \beta)$ that descends to cohomology.\\
Analogously to the proofs presented in \cite{bott2013differential,paradanverne2007relative}, we obtain the following
\begin{proposition}\label{prop rel cohom LES}
There is a natural long exact sequence in equivariant basic cohomology
\begin{equation} \cdots \stackrel{\alpha^*}{\to} H_G^k(N, A, \mathcal{E}) \stackrel{\beta^*}{\to} H_G^k(N, \mathcal{E}) \stackrel{\iota_A^*}{\to} H_G^k(A, \mathcal{E}) \to \cdots,\label{eq LES rel cohom} \end{equation}
where $\alpha^*(\theta)=(0,\theta)$, $\beta^*(\omega, \theta)=\omega$, and $\iota_A: A \to N$ denotes the inclusion.
\end{proposition}
\begin{remark}
	The complex $C^k_G(N,A,\mathcal{E})$ is a special case of the more general concept of a \emph{mapping cone} of a map of chain complexes (cf., e.g., \cite[Section~1.5]{weibel1997introduction}). In this context, the previous proposition corresponds to \cite[1.5.1]{weibel1997introduction}.
\end{remark}

The considerations of \cite[Section~3.2]{paradanverne2007relative} carry over to the basic setting so that we also obtain an excision statement for open submanifolds.
\begin{proposition}\label{prop excision}
	Let $A\subset N$ be a $G$-invariant, $\mathcal{E}$-saturated \emph{open} submanifold and $U$ a $G$-in\-va\-ri\-ant, $\mathcal{E}$-saturated open neighborhood of $N \setminus A$. Then the restriction to $(U, U\setminus (N \setminus A))$, $(\alpha, \beta) \mapsto (\alpha|_U, \beta|_{U\setminus(N\setminus A)})$, induces an isomorphism
	\[ H_G^k(N, A, \mathcal{E}) \to H_G^k(U, U\setminus (N \setminus A), \mathcal{E}).\]
\end{proposition}
\begin{proposition}\label{prop homotopy equiv relative}
	Let $A\subset N$ be any $G$-invariant, $\mathcal{E}$-saturated submanifold. If the equivariant maps $f: (N, A)\to (\tilde N, \tilde A)$ and $g:(\tilde N, \tilde A)\to (N, A)$ are $G \times \{\psi_t\}$-homotopy inverses, then they induce isomorphisms $f^*$ and $g^*$ in relative equivariant basic cohomology.
\end{proposition}
\begin{proof}
	Consider an equivariant homotopy $F: N \times I \to N$, $F(\cdot, 0)=g \circ f$, $F(\cdot, 1)=\operatorname{id}_N$ such that $F(A\times I)\subset A$. Then $F|_{A\times I}$ is a homotopy between $g \circ f|_A$ and $\operatorname{id}_A$. With $Q: C^k_G(N\times I, \mathcal{E})\to C^{k-1}_G(N, \mathcal{E})$, $\alpha \mapsto \int_0^1 \iota_{\partial_t} \alpha \ dt$, we then obtain (cf. \cite[\S~4]{bott2013differential} and \cite[Section~2.3.3]{guillemin1999supersymmetry})
	\begin{align}
		d_G QF^* + QF^*d_G = \operatorname{id}_N^*-f^*g^*. \label{eq chain homotopy}
	\end{align}
	With Equation \eqref{eq chain homotopy}, we can then show that $\operatorname{id}_N^*=f^*g^*$ in relative equivariant basic cohomology. Analogously, we obtain $\operatorname{id}_{\tilde N}^*=g^*f^*$ in relative equivariant basic cohomology, which yields that $f^*$ and $g^*$ are isomorphisms in relative equivariant basic cohomology.
\end{proof}
Note that the proofs (cf. also \cite{paradanverne2007relative}) of the previous propositions \ref{prop rel cohom LES}-\ref{prop homotopy equiv relative}
carry over to manifolds $N$ with invariant boundary and $A \subset N$ invariant open submanifold with invariant boundary, as long as the closure of $N\setminus A$ admits arbitrarily small invariant tubular neighborhoods. Propositions \ref{prop rel cohom LES} and \ref{prop homotopy equiv relative} 
also hold for manifolds $N$ with invariant boundary and $A \subset N$ invariant closed submanifold that is either $A\subset \mathrm{int}N$ or $A=\partial N$.\\
For compact manifolds $N$ and closed $G\times T$-invariant submanifolds $A\subset N$ (without boundary or with boundary as above), we have an alternative definition of relative cohomology (cf. \cite[Chapter~11.1]{guillemin1999supersymmetry}).
\begin{definition}\label{def alt rel cohom}
	Let $(N, \mathcal{E})$ be a compact foliated manifold with $G \times T$-action such that $\xi$ is nowhere vanishing and induces $\mathcal{E}$. Let $A\subset N$ be a closed $G\times T$-invariant submanifold. Assume that either $N$ is a manifold without boundary or that $N$ is a manifold with boundary such that $\partial N$ is $G\times T$-invariant, admits arbitrarily small invariant tubular neighborhoods and $A \subset \mathrm{int}N$ or $A=\partial N$. We define the complex $\widetilde C_G(N, A, \mathcal{E})$ to be the kernel of the pullback $C_G(N,\mathcal{E}) \to C_G(A, \mathcal{E})$. Since the pullback commutes with the differential, $\widetilde C_G(N, A, \mathcal{E})$ is a differential subcomplex of $C_G(N, \mathcal{E})$. We denote its cohomology by $\widetilde H_G(N, A, \mathcal{E})$.
\end{definition}
\begin{proposition}\label{prop rel cohom are isom}
	The map $\varphi: \widetilde C_G^k(N, A, \mathcal{E}) \to  C_G^k(N, A, \mathcal{E})$, $\omega \mapsto (\omega, 0)$ induces an isomorphism in cohomology.
\end{proposition}
\begin{proof}
	The map $\varphi: \widetilde C_G^k(N, A, \mathcal{E}) \to  C_G^k(N, A, \mathcal{E})$, $\omega \mapsto (\omega, 0)$ satisfies $D\circ \varphi = \varphi \circ d_G$ and, hence, induces a map in cohomology.\\
	Let $\pi: U \to A$ denote a $G\times T$-invariant tubular neighborhood and $f: N \to \mathbb{R}$ an invariant function with $\operatorname{supp} f \subset U$ and $f|_{\widetilde U} \equiv 1$ on a smaller neighborhood $\widetilde U$ of $A$. Then $\omega := f \pi^* \theta$ extends $\theta$ to $N$. The inclusion $\iota^*:\widetilde C_G^k(N, A, \mathcal{E}) \to C_G^k(N, \mathcal{E})$ is obviously injective. Hence, with $\iota_A^*$ denoting the pullback to $A$, we have a short exact sequence
	\[ 0\to \widetilde C_G^k(N, A, \mathcal{E}) \stackrel{\iota^*}{\to} C_G^k(N, \mathcal{E}) \stackrel{\iota_A^*}{\to} C_G^k(A, \mathcal{E}) \to 0,\] that induces a long exact sequence in cohomology: 
	\begin{equation} \cdots \to \widetilde H_G^k(N, A, \mathcal{E}) \to H_G^k(N, \mathcal{E}) \to H_G^k(A, \mathcal{E}) \to \cdots.\label{eq LES alternative rel cohom} \end{equation}
	Consider the following diagram, where the two horizontal sequences are sections of the two long exact sequences \ref{eq LES rel cohom} and \ref{eq LES alternative rel cohom} and, hence, exact.
	\[\begin{xy}
	\xymatrix@-0pt{
	H_G^{k-1}(N, \mathcal{E}) \ar[r]^{\iota_A^*} \ar[d]^{\operatorname{-id}} & H_G^{k-1}(A, \mathcal{E})\ar[r]^{\partial} \ar[d]^{\operatorname{-id}} & \widetilde H_G^k(N, A, \mathcal{E}) \ar[r]^{\iota^*} \ar[d]^{\varphi} & H_G^k(N, \mathcal{E}) \ar[r]^{\iota_A^*} \ar[d]^{\operatorname{id}} & H_G^k(A, \mathcal{E}) \ar[d]^{\operatorname{id}}\\
	H_G^{k-1}(N, \mathcal{E}) \ar[r]^{\iota_A^*} & H_G^{k-1}(A, \mathcal{E})\ar[r]^{\alpha^*} & H_G^k(N, A, \mathcal{E}) \ar[r]^{\beta^*} & H_G^k(N, \mathcal{E}) \ar[r]^{\iota_A^*} & H_G^k(A, \mathcal{E})\\
	}
	\end{xy}\]
	We want to apply the 5-Lemma. The leftmost square and the two squares on the right obviously commute. We show the commutativity of the remaining square; since $\pm \operatorname{id}$ is an isomorphism, the 5-lemma then yields that $\varphi$ is an isomorphism, as well. First, we determine the boundary operator $\partial$. Let $\theta$ represent a class in $H_G^{k-1}(A, \mathcal{E})$ and let $f$ be a cutoff function as above. Then it can easily be seen that $\partial \theta = df \wedge \pi^* \theta$. Further, we have $D(f\pi^* \theta,0)=(df\wedge \pi^* \theta,(f\pi^* \theta)|_A)= (df\wedge \pi^* \theta, \theta)$, so $(df\wedge \pi^* \theta,0)$ and $(0,- \theta)$ represent the same relative cohomology class. It follows that $\alpha^*(-\operatorname{id}(\theta))=(0,- \theta)= (df\wedge \pi^* \theta,0)= \varphi \circ \partial (\theta)$, the diagram commutes. 
\end{proof}

The proof of \cite[Theorem~11.1.1]{guillemin1999supersymmetry} is equally applicable in the basic (see also \cite{CF17loc}) and boundary setting so that we obtain
\begin{proposition}\label{prop rel cohom and cohom comp supp isom}
	Let $(N,A,\mathcal{E})$ be as in Definition \ref{def alt rel cohom}. The natural inclusion map $C_{G,c}(N\setminus A,\mathcal{E}) \to \widetilde C_G(N,A,\mathcal{E})$ induces an isomorphism in cohomology.
\end{proposition}

 Now, consider a $G\times T$-invariant closed submanifold $A \subset M$ of codimension $d$. Let $p: \nu A \to A$ denote the projection of the normal bundle.
A \emph{basic equivariant Thom form} is a closed form $\tau \in C^d_{G,c}(\nu A, \CF)$ satisfying $p_\ast \tau = 1$, with $p_\ast: C_{G,c}^k(\nu A,\CF) \to C_G^{k-d}(A,\CF)$ denoting fibrewise integration. A basic equivariant Thom form can be constructed analogously to \cite[Chapter~10]{guillemin1999supersymmetry} with an invariant basic connection form (see also \cite{CF17loc}). Note that a $G$-invariant basic connection form $\theta$ has to exist: By \cite[Proposition~2.8]{molino1988riemannian}, there always exists a connection that is \emph{adapted} to the lifted foliation, i.e., such that the tangent spaces to the leaves are horizontal. Since $G\times T$ is compact, we can obtain a $G\times T$-invariant adapted connection form by averaging over the group. But this connection form then has to be basic, or, as Molino calls it, \emph{projectable}.\label{existence basic connection}\\
Analogously, we can restrict ourselves to $\nu^\pm A$ instead if $A$ is a non-degenerate submanifold.\\[10pt]
In Section \ref{subsec cont mom map}, we scrutinized the functions $\Psi^{X_{s+1}}|_{Y_s}$ and the connected components $N$ of their critical sets. Recall that every $N$ is a $G \times T$-invariant closed submanifold of even codimension (cf. Lemma \ref{lem tot geod}) and non-degenerate (cf. Proposition \ref{prop Hess formula}). 
We will now consider the special case that  $A=N$. Denote the Morse index of $\Psi^{X_{s+1}}|_{Y_s}$ on $N$ by $\lambda$, the inclusion as the zero section $N \to \nu^\pm N$ by $\iota^\pm$ and the projection by $p^\pm: \nu^\pm N\rightarrow N$. 
For the following definition, compare \cite[Section~A.1]{GNTlocalization}.
\begin{definition}
	Let $k$ denote the rank of the (positive/negative) normal bundle $\nu^{(\pm)} N$. Then the bundle $P$ of oriented orthonormal frames of $(\nu^{(\pm)} N, \CF)$ is a foliated $SO(k)$-bundle over $N$. The equivariant basic Euler form $e_G(\nu^{(\pm)} N, \CF)\in C_G^*(N, \CF)$ of  $(\nu^{(\pm)} N, \CF) \rightarrow (N, \CF)$ is defined by
	\[e_G(\nu^{(\pm)} N, \CF)(X)=\operatorname{Pf}(F^\theta_G(X))=\operatorname{Pf}(F^\theta -\iota_X \theta), \]
	where $\theta \in \Omega^1(P, \CF)^G\otimes \mathfrak{so}(k)$ denotes a $G$-invariant basic connection form  on P, $F^\theta_G= d_G\theta + \tfrac{1}{2} [\theta,\theta]$ its equivariant curvature form and $\operatorname{Pf}$ the Pfaffian. 
	
	For any $G \times \{\psi_t\}$-invariant connection form, we can analogously define the equivariant Euler form $e_{\g \times \R \xi}(\nu^{(\pm)} N) \in C_{\g \times \R \xi}(N)$ or, for a $G\times T$-invariant connection form, the equivariant Euler form $e_{G \times T}(\nu^{(\pm)} N)$ $\in C_{G \times T}(N)$.
\end{definition}

Note that, while the Euler \emph{form} depends of the choice of connection form, its class (for which we use the same notation) does not. We can think of $e_{\g \times \R \xi}(\nu^{(\pm)} N)$ as the restriction of the polynomial map $e_{G \times T}(\nu^{(\pm)} N)$ to $\g \times \R \xi$.

\begin{proposition}\label{prop euler equal}
	Under the $S(\g^*)$-algebra isomorphism $H_G(M,\CF)= H_{\mathfrak{g}\times \R \xi}(M)$ of Proposition \ref{prop H G (M,F)=H Gx xi}, 
	\[e_G(\nu^{(\pm)} N, \CF)=e_{\g \times \R \xi}(\nu^{(\pm)} N).\]
\end{proposition}
\begin{proof}
	This becomes evident when regarding the inclusion of complexes that induces the isomorphism.
\end{proof}

Analogously to \cite[Theorem~10.6.1]{guillemin1999supersymmetry}, we obtain
\begin{theorem}[Basic equivariant Thom isomorphism]\label{theorem thom iso}
	Integration over the fiber defines an isomorphism
	\[p^-_*: H^{* + \lambda}_{G,c}(\nu^- N, \CF) \rightarrow H^*_G(N,\CF) \]
	whose inverse is given by the composition
	\[\iota^-_*:H^*_G(N,\CF) \xrightarrow{(p^-)^*}H^{* }_G(\nu^- N, \CF)\xrightarrow{\wedge \tau} H^{* + \lambda}_{G,c}(\nu^- N, \CF).\]
	
	As in \cite[Section~10.5]{guillemin1999supersymmetry}, it can be shown that $(\iota^-)^*\tau=e_G(\nu^- N,\CF) $ and, hence, that the composition $(\iota^-)^*\iota^-_*=\wedge e_G(\nu^- N,\CF)$ is the multiplication with the basic equivariant Euler class of $\nu^- N$. 
	
	The analogous statements hold for the whole and the positive normal bundle, with $\lambda$ replaced by $\operatorname{rank}(\nu N)$ and $\operatorname{rank}(\nu N)-\lambda$, respectively.	
\end{theorem}

\section{Basic Kirwan Surjectivity}\label{sec basic kirwan surj}
We will now proceed to state and prove our main result.

\begin{theorem}\label{theorem main}
Let $(M, \alpha)$ be a compact $K$-contact manifold, $\xi$ its Reeb vector field and $\CF$ the foliation that is induced by $\xi$. Let $G$ be a torus that acts on $M$, preserving $\alpha$. Denote with $\Psi \colon M \rightarrow \g^*$ the contact moment map and suppose that $0$ is a regular value of $\Psi$. Then the inclusion $\Psi^{-1}(0)\subset M$ induces an epimorphism in equivariant basic cohomology
\[ H^*_G(M,\CF) \longrightarrow H^*_G(\Psi^{-1}(0), \CF).\]
\end{theorem}
\begin{proof}
 Choose a metric $g$ adapted to $\alpha$ according to Lemma \ref{lem GxTinv metric}. Let $(X_1, ..., X_r)$ be a basis of $\mathfrak{g}$ according to Proposition \ref{prop basis K-contact}. Let again $f_s :=(\Psi^{X_1}, ... , \Psi^{X_s}) \colon M \rightarrow \R^s$, $Y_0 := M$ and $Y_s := f_s^{-1}(0)$ for $s=1, ..., r$. By Proposition \ref{prop Hess formula}, the functions $\Psi^{X_{s+1}}|_{Y_s}$ are Morse-Bott functions. 
We will show that the restrictions to the subsets $Y_{s+1} \subset Y_s$ induce the following sequence of epimorphisms:
\[ H^*_G(M, \CF)=H^*_G(Y_{0}, \CF) \to H^*_G(Y_{1}, \CF) \rightarrow ... \rightarrow H^*_G(Y_{r}, \CF)= H^*_G(\Psi^{-1}(0), \CF).\]

Set $Y_s^c := \left(\Psi^{X_{s+1}}|_{Y_s}\right)^{-1}((- \infty, c])$. Let $\kappa$ be a critical value of $\Psi^{X_{s+1}}|_{Y_s}$. We denote with $B^\kappa_{1}, ..., B^\kappa_{j_\kappa}$ the connected components of the critical set at level $\kappa$ and with $\lambda^\kappa_{i}$ the indices of the non-degenerate critical submanifolds $B^\kappa_i$ with respect to  $\operatorname{Hess}(\Psi^{X_{s+1}}|_{Y_s})$ and with $\nu^\pm B^\kappa_i$ their positive resp. negative normal bundles. 

Recall the following definition (cf. \cite[pp.~146f]{wasserman1969equivariant}).
\begin{definition}
	Let $V, W$ be Riemannian $G\times T$-vector bundles over $B^\kappa_i$. We denote their disk bundles by $D$, their open unit ball bundles by $\mathring{D}$ and their sphere bundles by $S$. The bundle $D^V \oplus D^W = \{(v,w)\in V \oplus W \mid ||v||\leq 1, ||w|| \leq 1 \}$ is called a \emph{handle bundle of type $(V,W)$} with index equal to the rank of $W$. Let $N \subset \tilde N$ be $G\times T$-manifolds with boundary, and $H \subset \tilde N$ a closed subset. We write 
	$\tilde N = N \cup_{D^V\oplus S^W} H$ and say that $\tilde N$ arises from $N$ by attaching a handle bundle of type $(V,W)$ if
	\begin{enumerate}[(i)]
		\item $\tilde F: D^V \oplus D^W \to H \subset \tilde N$ is a homeomorphism onto $H$,
		\item $\tilde N = N \cup H$,
		\item $\tilde F|_{D^V\oplus S^W}$ is an equivariant diffeomorphism onto $H \cap \partial N$,
		\item $\tilde F|_{D^V\oplus \mathring{D}^W}$ is an equivariant diffeomorphism onto $\tilde N \setminus N$.
	\end{enumerate}
\end{definition}
Let $\epsilon$ be small enough such that the interval $[\kappa - \epsilon, \kappa + \epsilon]$ contains no critical values of $\Psi^{X_{s+1}}|_{Y_s}$ besides $\kappa$. Since $\Psi^{X_{s+1}}|_{Y_s}$ is a $G \times T$-invariant Morse-Bott function, $Y_s^{\kappa + \epsilon}$ is $(G\times T)$-equivariantly diffeomorphic to $Y_s^{\kappa - \epsilon}$ with $j_\kappa$ handle bundles of type $(\nu^+B^\kappa_i, \nu^-B^\kappa_i)$ disjointly attached, cf. \cite[Theorem~4.6]{wasserman1969equivariant}.
\begin{align}
Y_s^{\kappa + \epsilon} \simeq Y_s^{\kappa - \epsilon} \cup_{\bigcup D^{\nu^+B^\kappa_i}\oplus  S^{\nu^- B^\kappa_{i}} } \bigcup D^{\nu^+B^\kappa_i}\oplus D^{\nu^-B^\kappa_{i}}. \label{eq homotopy type}
\end{align}

Here, the $G\times T$-action on $\nu^\pm B^\kappa_i$ is the natural lift of the $G \times T$-action on $M$. We denote the foliation induced by $\xi$ on the normal bundle also by $\CF$.
Let $U^\kappa _i$ denote an invariant tubular neighborhood of $ D^{\nu^+B^\kappa_i}\oplus  D^{\nu^- B^\kappa_{i}}$. By Diffeomorphism \eqref{eq homotopy type} and Proposition \ref{prop homotopy equiv relative}, we have
\begin{align}
 H_G^*(Y_s^{\kappa + \epsilon}, &Y_s^{\kappa - \epsilon}, \CF) =H_G^*(Y_s^{\kappa - \epsilon} 
 \cup_{\cup D^{\nu^+B^\kappa_i}\oplus  S^{\nu^- B^\kappa_{i}} } D^{\nu^+B^\kappa_i}\oplus D^{\nu^-B^\kappa_{i}}, Y_s^{\kappa - \epsilon}, \CF) \hspace{-41pt}&  \notag\\
 &= H^*_{G}(\cup \ U^\kappa_i , \cup \ U^\kappa_i \setminus D^{\nu^+B^\kappa_i}\oplus \mathring{D}^{\nu^-B^\kappa_{i}}, \CF) &\textrm{(by Prop. \ref{prop excision})} \notag\\
 &=H^*_{G}(\cup \ D^{\nu^+B^\kappa_i}\oplus  D^{\nu^- B^\kappa_{i}} , \cup \ D^{\nu^+B^\kappa_i}\oplus  S^{\nu^- B^\kappa_{i}}, \CF) &\textrm{(by Prop. \ref{prop homotopy equiv relative})} \notag\\
 &=H^*_{G}(\cup \  D^{\nu^- B^\kappa_{i}} , \cup \  S^{\nu^- B^\kappa_{i}}, \CF) &\textrm{(by Prop. \ref{prop homotopy equiv relative})} \notag\\
 &= \bigoplus H^*_{G}( D^{\nu^- B^\kappa_{i}} , S^{\nu^- B^\kappa_{i}}, \CF) & \notag\\
 &= \bigoplus \widetilde H^*_{G}( D^{\nu^- B^\kappa_{i}} , S^{\nu^- B^\kappa_{i}}, \CF)  &\textrm{(by Prop. \ref{prop rel cohom are isom})} \notag\\
 &= \bigoplus H^*_{G,c}( \mathring{D}^{\nu^- B^\kappa_{i}}, \CF) &\textrm{(by Prop. \ref{prop rel cohom and cohom comp supp isom})} \notag
\end{align}
Consider the $G\times T$-equivariant diffeomorphism $\rho: \mathring{D}^{\nu^- B^\kappa_{i}} \to \nu^- B^\kappa_{i}$, $v \mapsto \tfrac{1}{1-||v||^2}v$. Since $\rho$ is proper, Proposition \ref{prop homotopic same cohom} yields
\begin{align}
	H_G^*(Y_s^{\kappa + \epsilon}, Y_s^{\kappa - \epsilon}, \CF) &=\bigoplus H^*_{G,c}(\nu^- B^\kappa_i, \CF).\label{eq rel coh Ys is equal coh nu- comp supp}
\end{align}

By the Thom isomorphism (Theorem \ref{theorem thom iso}), we have further
\begin{align}
 H^{*-\lambda^\kappa_i}_G(B^\kappa_i, \CF) \stackrel{\sim}{\longrightarrow}H^*_{G,c}(\nu^- B^\kappa_i, \CF). \label{eq thom}
\end{align}

With 
$(G \times \{\psi_t\})_{B_i^\kappa}$ 
we denote the isotropy group of 
$G \times \{\psi_t\}$ 
on $B^\kappa_i$. Since $\Psi$ and $g$ are $G \times T$-invariant, $(G \times \{\psi_t\})_{B_i^\kappa}$ acts fiberwise on $\nu^-B_i^\kappa$ (and $D^{\lambda^\kappa_{i}} B^\kappa_i$) by restriction of the isotropy representation. We need the following lemmata.
\begin{lemma}\label{lem no non-zero fixed vectors}
	$\nu^{(\pm)} B_i^\kappa$ has no non-zero $(G \times \{\psi_t\})_{B_i^\kappa}$-fixed vectors.
\end{lemma}
\begin{proof}
For $x \in B_i^\kappa$, let $\gamma_v$ be the unique geodesic with initial values $\gamma_v(0)=x, \dot \gamma_v (0) = v$, $v \in \nu^{(\pm)}_x B_i^\kappa$. Since $G \times \{\psi_t\}$ acts by isometries, $g\cdot \gamma_v$ is again a geodesic and, by uniqueness, $g\cdot \gamma_v = \gamma_{dg(v)}$ for all $g \in (G \times \{\psi_t\})_{B_i^\kappa}$. Assume $v$ to be a $(G \times \{\psi_t\})_{B_i^\kappa}$-fixed vector. Then $g \cdot \gamma_v = \gamma_{dg(v)}=\gamma_v$ for all $g \in (G \times \{\psi_t\})_{B_i^\kappa}$, hence the isotropy group of all points along $\gamma_v$ contains $(G \times \{\psi_t\})_{B_i^\kappa}$. By Lemma \ref{lem crit= min Gpsi (GT)-orbits}, however, the critical set is the union of all minimal $G\times \{\psi_t\}$-orbits, hence $\gamma_v$ lies completely in the connected component $B_i^\kappa$. Thus $v = \dot \gamma_v (0) \in T_x B_i^\kappa \ \bot \ \nu_x B_i^\kappa$, therefore $v=0$.
\end{proof}

It follows that $\nu^\pm B_i^\kappa$ has no non-zero $(G \times T)_{B_i^\kappa}$-fixed vectors, therefore, the  multiplication with the Euler classes of the negative, positive or whole normal bundle in $H^{*}_{G\times T}(B_i^\kappa)$ is injective (see \cite[Proposition~5]{duflot1983smooth} or \cite[\S~13]{atiyah1983yang}). 
We now show that this also holds for their restriction to $\g \times \R \xi$. 
\begin{lemma}\label{lem wedge E injective}
	Multiplication in $H^{*}_{\g \oplus \R \xi}(B_i^\kappa)$ with the  
	equivariant Euler class of the negative, positive or whole normal bundle of $B_i^\kappa$ is injective.
\end{lemma}
\begin{proof}
We present the proof for the case of the negative normal bundle, the other cases work analogously. Denote the Euler class of $\nu^-B_i^\kappa $ by $E_i^\kappa$. 
Let $\theta$ denote a $G\times T$-invariant connection 1-form in the bundle $P$ of oriented orthonormal frames of the negative normal bundle over $B_i^\kappa$. 
Then by definition, for $X\in \g\oplus \R \xi$, $E^\kappa_i(X)$ is given by 
$\mathit{Pf}(F^\theta -  \iota_{X} \theta)$, where we denote the Pfaffian $\in S(\mathfrak{so} (\lambda_i^\kappa)^* )^{SO( \lambda_i^\kappa ) }$ by $\mathit{Pf}$. The classification of irreducible torus representations yields that $\nu^-B_i^\kappa$ splits into 2-dimensional subbundles s.t., when written in a basis adapted to the splitting, the $(\g \oplus \mathfrak{t})_{B_i^\kappa}$-action is given by the standard action of the 
matrix 
\begin{align} \left( \begin{smallmatrix}
   	0 & -\alpha_1(X) & & &\\
   	\alpha_1(X) & 0 & & &\\
   	 & & \ddots & &\\
   	  & & & 0 & -\alpha_{\lambda_i^\kappa /2}(X)\\
   	  & & & \alpha_{\lambda_i^\kappa /2}(X) & 0
   \end{smallmatrix}\right), \quad X \in (\g \oplus \mathfrak{t})_{B_i^\kappa} \label{eq weight matrix}\end{align}
with the weights $\alpha_1, ..., \alpha_{\lambda_i^\kappa /2}$ of the $(\g \oplus \mathfrak{t})_{B_i^\kappa}$-representation. For every $X\in (\g \oplus \mathfrak{t})_{B_i^\kappa}$, Matrix \eqref{eq weight matrix} is an element of $\mathfrak{so}(\lambda_i^\kappa)$. Thus, $X_P$ and the $SO$-fundamental vectorfield generated by Matrix \eqref{eq weight matrix} coincide. By the definition of a connection form, $\theta (Y_P) = Y$ for every $Y \in \mathfrak{so}$. Therefore, it holds for every $X \in  (\g \oplus \mathfrak{t})_{B_i^\kappa}$ that
\begin{equation*}
	\iota_X \theta = \left( \begin{smallmatrix}
   	0 & -\alpha_1(X) & & &\\
   	\alpha_1(X) & 0 & & &\\
   	 & & \ddots & &\\
   	  & & & 0 & -\alpha_{\lambda_i^\kappa /2}(X)\\
   	  & & & \alpha_{\lambda_i^\kappa /2}(X) & 0
   \end{smallmatrix}\right).
\end{equation*}
Since $(\g \oplus \R \xi)_{B_i^\kappa}\subset (\g \oplus \mathfrak{t})_{B_i^\kappa}$, we obtain for every $X \in (\g \oplus \R \xi)_{B_i^\kappa}$
\begin{equation}
	\mathit{Pf}(\iota_X \theta) = \frac{1}{\left(-2\pi\right)^{\lambda_i^\kappa /2}} \prod_{j=1}^{\lambda_i^\kappa /2} \alpha_j(X). \label{eq Pf iXtheta}
\end{equation}
 Let $\mathfrak{k}$ be a complement of $(\mathfrak{g} \oplus \R \xi)_{B_i^\kappa}$ in $\mathfrak{g} \oplus \R \xi$. Then, by the definition of the Cartan complex, we have $C_{\mathfrak{g} \oplus \R \xi}(B_i^\kappa) =S((\mathfrak{g} \oplus \R \xi)_{B_i^\kappa}^*)\otimes C_\mathfrak{k}(B_i^\kappa)$, $d_{\mathfrak{g} \oplus \R \xi}=1\otimes d_{\mathfrak{k}}$, and $H_{\mathfrak{g} \oplus \R \xi }(B_i^\kappa) =S((\mathfrak{g} \oplus \R \xi)_{B_i^\kappa}^*)\otimes H_\mathfrak{k}(B_i^\kappa)$.  $S((\mathfrak{g} \oplus \R \xi)_{B_i^\kappa}^*)$ is a polynomial ring, so any $\omega_0 \in S((\mathfrak{g} \oplus \R \xi)_{B_i^\kappa}^*)$ with $\omega_0 \neq 0$ is not a zero divisor in $H_{\mathfrak{g} \oplus \R \xi }(B_i^\kappa)$. More generally, if there is an $\omega_0 \in S((\mathfrak{g} \oplus \R \xi)_{B_i^\kappa}^*)$ such that  $\omega \in H_{\mathfrak{g} \oplus \R \xi }(B_i^\kappa)$ is of the form 
\[\omega = \omega_0 \otimes 1 + \text{terms of positive degree in }H_\mathfrak{k}(B_i^\kappa),\]
then $\omega$ is not a zero divisor in $H_{\mathfrak{g} \oplus \R \xi }(B_i^\kappa)$ (cf. also \cite[p.~605]{atiyah1983yang}). Hence, for $E_i^\kappa$ not to be a zero divisor, it suffices to show that its purely polynomial part in $S((\mathfrak{g} \oplus \R \xi)_{B_i^\kappa}^*)\otimes 1$ is not a zero divisor. Since $E^\kappa_i$ is a form of degree $\lambda^\kappa_i$, as is $\prod_{j=1}^{\lambda_i^\kappa /2} \alpha_j$, it follows with Equation \eqref{eq Pf iXtheta} that 
\[  E^\kappa_i = \frac{1}{\left(2\pi\right)^{\lambda_i^\kappa /2}} \prod_{j=1}^{\lambda_i^\kappa /2} \alpha_j \otimes 1 + \text{terms of positive degree in }H_\mathfrak{k}(B_i^\kappa).\]
Thus, it suffices to show that $\prod_{j=1}^{\lambda_i^\kappa /2} \alpha_j \not\equiv 0$ on $(\mathfrak{g} \oplus \R \xi)_{B_i^\kappa}$.
Suppose that $\prod_{j=1}^{\lambda_i^\kappa /2} \alpha_j$ vanishes on $(\mathfrak{g} \oplus \R \xi)_{B_i^\kappa}$. Then there exists an $\alpha_{j_0}$ that vanishes on all of  $(\mathfrak{g} \oplus \R \xi)_{B_i^\kappa}$. By Matrix \eqref{eq weight matrix}, this means that a two-dimensional subspace of $\nu^-B_i^\kappa$ vanishes under $(\mathfrak{g} \oplus \R \xi)_{B_i^\kappa}$. This, however, contradicts Lemma \ref{lem no non-zero fixed vectors}.
\end{proof}
Recalling Propositions \ref{prop H G (M,F)=H Gx xi} and \ref{prop euler equal}, we also set $E_i^\kappa = e_G(\nu^- B_i^\kappa, \CF)$ by abuse of notation. We obtain an injective map
\begin{align}
 \oplus ( \cdot E^\kappa_i) \colon \bigoplus_i H^{*-\lambda^\kappa_i}_{G}(B^\kappa_i,\CF) \longrightarrow \bigoplus_i H^{*}_{G}(B^\kappa_i,\CF). \label{eq Eul}
\end{align}

Now, set $Y_s^\pm := \{\pm \Psi^{X_{s+1}}|_{Y_s} \geq 0\}$. Obviously, we then have $Y_{s+1} = \{\Psi^{X_{s+1}}|_{Y_s} = 0\}=Y_s^+ \cap Y_s^-$. Let $0< \kappa_0 < \kappa_1 < ...< \kappa_m$ be the critical values of $\Psi^{X_{s+1}}|_{Y_s}$ attained on $Y_s^+$. Consider the following diagram, in which the top row is the long exact sequence of the pair $((Y_s^+)^{\kappa_j + \epsilon_j}, (Y_s^+)^{\kappa_j - \epsilon_j})$, see Proposition \ref{prop rel cohom LES}.
Note that by excision and homotopy equivalence, we have $H_{G}^*(Y_s^{\kappa_j + \epsilon_j}, Y_s^{\kappa_j - \epsilon_j},\CF) = H_{G}^*((Y_s^+)^{\kappa_j + \epsilon_j}, (Y_s^+)^{\kappa_j - \epsilon_j},\CF)$.
The following argument is similar to that in \cite[Theorem~7.1]{goertsches2010torus}. The Isomorphisms \eqref{eq rel coh Ys is equal coh nu- comp supp} and \eqref{eq thom} yield that the two vertical arrows on the left are isomorphisms. The vertical arrow on the right and the diagonal arrow are the restriction to $B^{\kappa_j}_i$,the composition of $h_j$ and the right vertical arrow is  the restriction to $B^{\kappa_j}_i$ of the first factor. By Theorem \ref{theorem thom iso}, the diagram is commutative. Multiplication by $\oplus ( \cdot E^\kappa_i)$ is injective by \eqref{eq Eul}, therefore $h_j$ has to be injective. 
\[
\begin{xy}
 \xymatrix@-13pt{
    *+[r]{... } \ar[r] & H_{G}^*(Y_s^{\kappa_j + \epsilon_j}, Y_s^{\kappa_j - \epsilon_j},\CF) \ar[r]^{h_j} \ar[d]_{\cong} & H_{G}^*((Y_s^+)^{\kappa_j + \epsilon_j},\CF) \ar[r] \ar[dd] & H_{G}^*((Y_s^+)^{\kappa_j - \epsilon_j},\CF) \ar[r] &...\\ 
  & \bigoplus_i H^*_{G,c}(\nu^- B^{\kappa_j}_i,\CF) \ar[rd] \ar[d]_{\cong} &   & \\
  & \bigoplus_i H^{*-\lambda^{\kappa_j}_i}_{G}(B^{\kappa_j}_i,\CF) \ar[r]_{\oplus ( \cdot E^{\kappa_j}_i)} & \bigoplus_i H^{*}_{G}(B^{\kappa_j}_i,\CF). & 
}
\end{xy}
\]

By injectivity of $h_j$, the long exact sequence 
turns into the short exact sequences
\[ 0 \rightarrow H_{G}^*(Y_s^{{\kappa_j} + \epsilon_j}, Y_s^{\kappa_j - \epsilon_j},\CF) \rightarrow H_{G}^*((Y_s^+)^{\kappa_j + \epsilon_j},\CF) \xrightarrow{\iota_j} H_{G}^*((Y_s^+)^{\kappa_j - \epsilon_j},\CF)  \rightarrow 0,\]
hence, the natural map $\iota_j$ is surjective.  Furthermore, we know that the homotopy type does not change before crossing a critical value, cf. \cite[Theorem~4.3]{wasserman1969equivariant},  
thus $H_{G}^*((Y_s^+)^{\kappa_j - \epsilon_j},\CF) =H_{G}^*((Y_s^+)^{\kappa_{j-1}+ \epsilon_{j-1}},\CF)$. In particular, this implies $H_{G}^*((Y_s^+)^{\kappa_0 - \epsilon_0},\CF) =H_{G}^*(Y_{s+1},\CF)$ and  $H_{G}^*((Y_s)^+,\CF) = H_{G}^*((Y_s^+)^{\kappa_m + \epsilon_m},\CF)$. This yields the following sequence of surjective maps
\[ H_{G}^*((Y_s)^+,\CF) = H_{G}^*((Y_s^+)^{\kappa_m + \epsilon_m},\CF)\rightarrow \cdots \rightarrow H_{G}^*((Y_s^+)^{\kappa_0 + \epsilon_0},\CF) \rightarrow H_{G}^*(Y_{s+1},\CF).\]
Thus, the natural map $H_{G}^*(Y_s^+,\CF) \rightarrow H_{G}^*(Y_{s+1},\CF) $ is surjective. 

Analogous reasoning with $-\Psi^{X_{s+1}}|_{Y_s}$ and, hence, the Euler classes of the positive normal bundles yields the surjectivity of $ H_{G}^*(Y_s^-,\CF) \rightarrow H_{G}^*(Y_{s+1},\CF)$. 

These epimorphisms turn the Mayer-Vietoris sequence of $(Y_{s+1}, Y_s^+, Y_s^-)$ (see Proposition \ref{prop MVS}; more precisely, we apply Proposition \ref{prop MVS} to the two open sets $\{x \in Y_s \mid \pm \Psi^{X_{s+1}}(x) > -\delta \}$ which, for sufficiently small $\delta >0$, are $G\times T$-homotopy equivalent to $Y_s^\pm$) into the short exact sequences
\begin{align}
 0 \rightarrow H_{G}^*(Y_s,\CF) \stackrel{(j^+)^* \oplus (j^-)^*}{\rightarrow} H_{G}^*(Y_s^+,\CF)\oplus H_{G}^*(Y_s^-,\CF)\stackrel{(i^+)^* - (i^-)^*}{\to} H_{G}^*(Y_{s+1},\CF)\rightarrow 0, \label{eq MVS i j}
 \end{align}
where $j^\pm \colon Y_s^\pm \hookrightarrow Y_s$ and $i^\pm \colon Y_{s+1} \hookrightarrow Y_s^\pm$ denote the inclusions. We claim that the composition of these maps 
induces an epimorphism in equivariant basic cohomology. So let $\omega \in H_{G}^*(Y_{s+1},\CF)$ be arbitrary. We know that $(i^\pm)^*$ are surjective, hence there exist $\eta^\pm \in H_{G}^*(Y_s^\pm,\CF)$ such that $(i^\pm)^*(\eta^\pm)=\omega$. But this means that $\eta^+ + \eta ^- \in \operatorname{ker} ((i^+)^*-(i^-)^*)= \operatorname{im} ((j^+)^* + (j^-)^*)$, i.e., there is a $\sigma \in H_{G}^*(Y_s,\CF)$ such that $ \eta^+ + \eta ^- =(j^+)^*(\sigma) + (j^-)^*(\sigma)$. This, however, yields $\omega = (i^+)^*\circ (j^+)^*(\sigma) = (i^-)^* \circ (j^-)^* (\sigma)$ and concludes the proof of the surjectivity
\[ H_{G}^*(Y_s,\CF) \twoheadrightarrow H_{G}^*(Y_{s+1},\CF).\]
Iteration for $s=0, ..., r-1$ yields the desired sequence of epimorphisms 
\[ H^*_G(M, \mathcal{F})=H^*_G(Y_{0}, \mathcal{F}) \rightarrow H^*_G(Y_{1}, \mathcal{F}) \rightarrow ... \rightarrow H^*_G(Y_{r}, \mathcal{F})= H^*_G(\Psi^{-1}(0), \mathcal{F}). \qedhere \]
\end{proof}

\begin{remark}
The idea to obtain the basic Kirwan map as the composition of surjective maps
$H^{*}_{G}(Y_s, \mathcal{F})\twoheadrightarrow H^{*}_{G}(Y_{s+1}, \mathcal{F}) $
stems from the approach used in \cite[Proof~of~Theorem~G.13]{ggk2002moment} and \cite[Proof~of~Theorem~3.4]{baird2010topology}. To obtain surjectivity, Thom and Euler class arguments were also used in \cite{baird2010topology}. 
\end{remark}

We want to point out that Equation \eqref{eq crit psi x} implies that the $G$-action on $\Psi^{-1}(0)$ is locally free if $0$ is a regular value of $\Psi$. Then $\Omega(\Psi^{-1}(0), \CF)$ is a $\mathfrak{g}$-dga of type (C) (cf. \cite[Section~5.1]{guillemin1999supersymmetry} and \cite[Proof of Lemma 3.18]{goertsches2010equivariant}) and Proposition \ref{prop H G = H(bas) if of type C} yields $H_G(\Psi^{-1}(0), \CF)\cong H(\Omega(\Psi^{-1}(0), \CF)_{\mathrm{bas }\mathfrak{g}})$. In case of a free $G$-action on $\Psi^{-1}(0)$, this implies that we have an isomorphism $H_G(\Psi^{-1}(0), \CF)\cong H(\Psi^{-1}(0)/G, \CF_0))$, where $\CF_0$ denotes the foliation induced by $\xi$ on the contact quotient $\Psi^{-1}(0)/G$.

\section{Examples}\label{sec ex}
\subsection{Boothby-Wang fibration}
This example shows that Theorem \ref{theorem main}, under the assumption that $\xi$ induces a free $S^1$-action, reproduces Kirwan's surjectivity result for the symplectic $S^1$-quotient. In this case, $\{\psi_t\}=S^1=T$ and 	$\pi \colon M \rightarrow M/\{\psi_t\} =: B$ is a $G$-equivariant 
	principal $S^1$-bundle. $d \alpha$ descends to a symplectic form $\omega$ on $B$, $d \alpha = \pi^* \omega$ (see, e.g.,~\cite[Theorem~6.1.26]{boyer2008sasakian}). Furthermore, we have $H_G(M, \CF) = H_G(B)$ (compare Example \ref{ex xi free S1 action} and \cite[Example~3.14]{goertsches2010equivariant}). A symplectic moment map $\mu$ on $B$ is defined up to a constant by $d(\mu^X) = \iota_{X_{B}} \omega$. Since $\mathcal{L}_X \alpha = 0$, however, this equation, when pulled back to $M$, is equivalent to $-d \pi^*\mu^X = d \iota_{X_M} \alpha$. $\iota_{X_M} \alpha$ is an $S^1$-invariant function, so there is a $f^X \in \Omega^0(B)$ such that $\pi^* f^X = \iota_{X_M} \alpha$. $\mu^X := - f^X$ then defines a moment map for the $G$-action on $(B, \omega)$. Then $\mu^{-1}(0) = \Psi^{-1}(0)/S^1$ and Theorem \ref{theorem main} yields the known Kirwan surjectivity induced by the inclusion $\mu^{-1}(0) \hookrightarrow B$.
	\[H^{*}_G(B)= H_G(M, \CF)\twoheadrightarrow H^{*}_{G}(\Psi^{-1}(0),\mathcal{F})=H^{*}_{G}(\Psi^{-1}(0)/S^1)= H^{*}_{G}(\mu^{-1}(0)). \]

\subsection{\texorpdfstring{$S^{1}$}{S{1}}-action on a Weighted Sphere}\label{subsection ex sphere}
	We will also present an example where $T \neq S^1$.
	
	Consider $(M, \alpha)=(S^{2n+1}, \alpha_w)$ from Example \ref{ex weighted sphere} with weight  $w \in \R^{n+1}$, $w_j >0$. If at least two $w_j$ are linearly independent over $\mathbb{Q}$, then $T$ is a torus of rank $\geq 2$.
	Then
	\[\alpha=\frac{\tfrac{i}{2} \left( \sum_{j=0}^n z_j d\bar z_j - \bar z_j d z_j \right)}{\sum_{j=0}^n w_j|z_j|^2} , \quad \xi = i\left( \sum_{j=0}^{n} w_j(z_j \tfrac{\partial}{\partial z_j} - \bar z_j \tfrac{\partial}{\partial \bar z_j}) \right).\]
	The flow of $\xi$ is given by $\psi_t (z) = (e^{itw_0}z_0, ..., e^{itw_n}z_{n})$. 	
	Furthermore, let $G=S^1$ act (freely) on $S^{2n+1}$ with weights $\beta = (1,...,1,-1)$, that is, by $\lambda \cdot z = (\lambda z_0, ...,\lambda z_{n-1},\lambda^{-1} z_n)$. The fundamental vector field $X$ corresponding to $1 \in \R \simeq \mathfrak{s}^1$ is given by
	\[X(z) = i\left( \sum_{j=0}^{n-1} (z_j \tfrac{\partial}{\partial z_j} - \bar z_j \tfrac{\partial}{\partial \bar z_j}) -z_n \tfrac{\partial}{\partial z_n} + \bar z_n \tfrac{\partial}{\partial \bar z_n}\right)\]
	and we compute the contact moment map to be 
	\[ \Psi (z) = \frac{ \sum_{j=0}^{n-1} |z_j|^2 -|z_n|^2}{\sum_{j=0}^n w_j|z_j|^2}=\frac{ 1-2|z_n|^2}{\sum_{j=0}^n w_j|z_j|^2}. \]
	Hence, we obtain 
	\[\Psi^{-1}(0) = \left\lbrace z \in S^{2n+1} \mid \sum_{j=0}^{n-1} |z_j|^2 = \tfrac{1}{2}=|z_n|^2\right\rbrace= S^{2n-1}\left(\tfrac{1}{\sqrt{2}}\right) \times S^{1}\left(\tfrac{1}{\sqrt{2}}\right).\]
	To compute $H_G(M, \CF)$, consider the standard diagonal $S^1$-action on $\C^{n+1}$ given by $\lambda \cdot z = (\lambda z_0, ..., \lambda z_n)$. This action is Hamiltonian (w.r.t. the standard symplectic structure on $\C^{n+1}$) 
	and a moment map is given by $\mu(z)=\tfrac{1}{2} \sum_j|z_j|^2$. Set $f:= || \mu - 1/2||^2$. The $G \times T$-action on $M$ can naturally be extended to $\C^{n+1}$. $\mu$ and, hence, $f$ are $G \times T$-invariant. We will compute $H_G(M, \CF)$ by applying Morse theory with $f$ on $\C^{n+1}$, an idea learned from Jonathan Fisher, see also \cite{kirwan1984cohomology}. The critical set of $f$ is given by $\operatorname{Crit}(f)= \{0\} \ \dot \cup \ M$, and the critical values are $f(0)=1/4$, $f(M)=0$. The Hessian $H$ of $f$ at 0 is given by $- \operatorname{Id}$, which is nondegenerate. For $z \in M$, the normal direction (to $M$) is spanned by $Y := \sum z_j \partial_{z_j}+ \bar z_j \partial_{\bar z_j}$ and $H_z(Y,Y)=2$, 
	which yields that $H_z$ is non-degenerate in normal direction. It follows that $f$ is a $G \times T$-invariant Morse-Bott function.
	Note that the relative $\g \times \R \xi$-equivariant cohomology is constructed as in Definition \ref{def rel eq coh} and, analogously to Proposition \ref{prop rel cohom LES}, we obtain the long exact sequence of the pair $(\{z \in \C^{n+1} \mid f(z)\leq 1/4+\epsilon\}, \{z \in \C^{n+1} \mid f(z)\leq 1/4-\epsilon\})$. Similarly, the relevant isomorphisms of Section \ref{subsec eq bas coh} can be transferred to this setting and, analogously to the Isomorphisms \eqref{eq rel coh Ys is equal coh nu- comp supp} and \eqref{eq thom}, we obtain the isomorphism 
	\begin{align*}
	H_{\g+\R \xi}^* (\{ f\leq 1/4+\epsilon\}, & \{ f\leq 1/4-\epsilon\}) \\
	&\cong 	H_{\g+\R \xi}^*(\{ f\leq 1/4-\epsilon\} \cup_{ S^{\nu\{0\}}} D^{\nu \{0\}}, \{ f\leq 1/4-\epsilon\})  \\
	&\cong 	H_{\g+\R \xi}^*(D^{\nu \{0\}},S^{\nu\{0\}})  \\
	&\cong 	H_{\g+\R \xi,c}^*(\mathring{D}^{\nu \{0\}})  \\
	&\cong 	H_{\g+\R \xi,c}^*(\nu \{0\})  \\
	&\cong 	H_{\g+\R \xi}^{*-2(n+1)}( \{0\}) 
	\end{align*}
		
	Note that $\{z \in \C^{n+1} \mid f(z)\leq 1/4+\epsilon\}$ is a $G \times T$-equivariant retraction of $\C^{n+1}$ and that $\{z \in \C^{n+1} \mid f(z)\leq 1/4-\epsilon\}$ $G \times T$-equivariantly retracts onto $M$. 
	We obtain the isomorphism 
	\[T: H_{\g+\R \xi}^*(\C^{n+1}, M)\cong H_{\g+\R \xi}^* (\{ f\leq 1/4+\epsilon\},  \{ f\leq 1/4-\epsilon\}) \cong H^{*-2(n+1)}_{\g+\R \xi}(\{0\}).\]
	The long exact sequence then looks as follows
	\[ \begin{xy}
 \xymatrix@-13pt{
    *+[r]{... } \quad \ar[r] &\quad H_{\g+\R \xi}^*(\C^{n+1}, M) \quad \ar[r] \ar[d]_{\cong}^{T} & \quad H_{\g+\R \xi}^*(\C^{n+1}) \ar[r]  & H_{\g+\R \xi}^*(M) \ar[r] &...\\ 
    & H^{*-2(n+1)}_{\g+\R \xi}(\{0\}) }
\end{xy}
	\]
	In this diagram, the combination of $T^{-1}$ with the restriction from $\C^{n+1}$ to $\{0\}$, that is,  $H_{\g+\R \xi}^{*-2(n+1)}(\{0\})\to H_{\g+\R \xi}^*(\C^{n+1})\stackrel{\cong}{\to}   H_{\g+\R \xi}^*(\{0\})$, is multiplication with the equivariant Euler class $e$ of the normal bundle to $\{0\}$, which is injective (this is seen analogously to Lemma \ref{lem wedge E injective}). We obtain short exact sequences
	\begin{align} 0 \to H_{\g+\R \xi}^{*-2(n+1)}(\{0\})\stackrel{\cdot e}{\to}   H_{\g+\R \xi}^*(\{0\}) \to H_{\g+\R \xi}^*(M) \to 0.\label{eq ses example sphere} \end{align}
	We will now compute $e$. The (negative) normal bundle is the trivial bundle $\C^{n+1} \times \{0\}$ which is the product of the line bundles $\nu_j := \C_j \times \{0\}$, where $\C_j$ denotes the $j$-th coordinate. 
	 The bundle of oriented orthonormal frames of $\nu_j$ is the trivial bundle $P_j=SO(2)\times \{0\}$. The canonical flat connection form $\theta_j$ on $P_j$ is invariant under $G\times \{\psi_t\}$. 
	 The vector fields generated by $X$ and $\xi$ on $P_j$ coincide with the fundamental vector fields of the $SO(2)$-action with weights $\beta_j$, $w_j$, respectively, so $\iota_X \theta_j = \left( \begin{smallmatrix} 0 & -\beta_j\\ \beta_j & 0                                                                                                                                                                                                                                                                                                                                                                                   \end{smallmatrix}\right)$ and $\iota_\xi \theta_j = \left( \begin{smallmatrix} 0 & -w_j\\ w_j & 0                                                                                                                                                                                                                                                                                                                         
                                                          \end{smallmatrix}\right)$. Since the curvature of $\theta_j$ is zero, the Euler class $e_j$ of $\nu_j$ is then given by $\mathrm{Pf}(-\iota_X\theta_j u - \iota_\xi \theta_j s)=\tfrac{1}{2\pi}(u\beta_j + sw_j)$, where $(u,s)$ are dual to $(X, \xi)$. We then obtain $e$ as $e=\prod_j e_j = \tfrac{1}{(2\pi)^{n+1}} (-u + w_ns)\prod_{j=0}^{n-1}(u +w_js)$.
	
	The short exact sequence \eqref{eq ses example sphere} then yields that, as $(\R[u]=S(\g^*))$-algebra,
	\[H_G(M, \CF)= \R[u,s]/\langle e \rangle = \R[u,s]\left/\left\langle (-u + w_ns)\prod_{j=0}^{n-1}(u +w_js) \right\rangle \right..\]
	
	$G$ acts freely on $\Psi^{-1}(0)$, so $H_G(\Psi^{-1}(0),\CF)=H(\Psi^{-1}(0)/G,\CF)$ by Proposition \ref{prop H G = H(bas) if of type C}, where we denote the induced foliation on the quotient also by $\CF$. $\Psi^{-1}(0)/G$ is $\{\psi_t\}$-equivariantly diffeomorphic to $S^{2n-1}(\tfrac{1}{\sqrt{2}})$ via $[z] \mapsto (\sqrt{2}z_n z_0, ..., \sqrt{2}z_n z_{n-1})$, where $\psi_t$ acts on $S^{2n-1}(\tfrac{1}{\sqrt{2}})$ by 
	$\psi_t(z)=(e^{it(w_0+w_n)}z_0, ..., e^{it(w_{n-1}+w_n)}z_n)$. 
	This is the Reeb flow of a weight\-ed Sasakian structure on $S^{2n-1}$, hence, the induced foliation is defined by the Reeb vector field belonging to this Sasakian structure. It follows that the basic cohomology ring $H^*(\Psi^{-1}(0)/G,\CF)$ is isomorphic to $H^*(\C P^{n-1})$, see \cite[Proposition~7.5.29]{boyer2008sasakian}. 
	
	We will now compute the restriction from $H_{\g+\R \xi}^*(M) $ to $H_{\g+\R \xi}^*(\Psi^{-1}(0)) $. Since $\iota_X \alpha$ vanishes on $\Psi^{-1}(0)$ and $\iota_\xi \alpha=1$, we have $0=[d_{\g+\R \xi} \alpha]=[d\alpha-s]$ in $H_{\g+\R \xi}^*(\Psi^{-1}(0)) $. Similarly, consider the $G\times T$-invariant 1-form 
	\[\gamma:=w_n \alpha - i(z_nd\bar z_n - \bar z_n dz_n) .\]
	On $\Psi^{-1}(0)$, we have $\gamma(\xi)=
	0$ and $\gamma(X)=
	1$, so $d_{\g+\R \xi}\gamma = d \gamma + u$. Since $d(i(z_nd\bar z_n - \bar z_n dz_n))$ vanishes on $\Psi^{-1}(0)$, it is $d\gamma= w_n d\alpha$. 
	It follows that $[u]=[w_n d\alpha]=[w_ns]$ in $H_{\g+\R \xi}^*(\Psi^{-1}(0)) $. Note that $\alpha|_{\Psi^{-1}(0)}$ is $G$-basic, so $d\alpha^{\wedge n}|_{\Psi^{-1}(0)}=0$.
	Under $[s]\mapsto [\tfrac{1}{w_n}u]$, $u^n\mapsto 0$,  $H_G(M, \CF)= \R[u,s]/\langle e \rangle$ is surjectively mapped to $\R[u]/\langle u^n \rangle \cong H(\C P^{n-1})$.

\section{The Kernel of the Kirwan Map}\label{sec kernel}
In this section, we derive a description of the kernel of the basic Kirwan map for $G=S^1$. Recall that $(M, \alpha)$ denotes a compact $K$-contact manifold, $\xi$ its Reeb vector field and $\Psi$ the contact moment map for the action of a torus $G$ on $M$ that preserves $\alpha$. We assume that 0 is a regular value of $\Psi$. Throughout this section, $(X_s)$ denotes a basis of $\g$ according to Proposition \ref{prop basis K-contact}. Recall that we set $Y_s = (\Psi^{X_1},..., \Psi^{X_s})^{-1}(0)$, $Y_0=M$, and $Y_s^\pm = \{ \pm \Psi^{X_{s+1}}|_{Y_s} \geq 0\}$. As in the previous section, denote the inclusions by $\iota_s : Y_{s+1} \rightarrow Y_s$, $\iota_s^\pm : Y_{s+1} \rightarrow Y_s^\pm$ and $j_s^\pm : Y_{s}^\pm \rightarrow Y_s$. Additionally, set $C_s := \mathit{Crit}(\Psi^{X_{s+1}}|_{Y_s})$. Recall that $C_1=\mathit{Crit}(\Psi)$, see Lemma \ref{lem crit= min Gpsi (GT)-orbits}.

We adjust the computations that Tolman and Weitsman did in the symplectic setting (\cite[Section~3]{tolman2003kernel}) to our case. Note that we apply the results they obtained for $S^1$-actions to the components $\Psi^{X_{s+1}}|_{Y_s}$ for actions of tori of arbitrary rank. The following Lemma corresponds to \cite[Lemma~3.1]{tolman2003kernel}.
\begin{lemma}\label{lem ses}
	Let $f= \Psi^{X_{s+1}}|_{Y_s} $ or $f= - \Psi^{X_{s+1}}|_{Y_s} $ and let $\kappa$ be any critical value of $f$. Denote with $B_i^\kappa$ the connected components of $C_s \cap f^{-1}(\kappa)=:C_s^\kappa$ and with $\lambda_i^\kappa$ their indices. Let $\epsilon >0$ such that $[\kappa- \epsilon, \kappa + \epsilon ]$ does not contain a critical value besides $\kappa$. Then there exists a short exact sequence
	\[0 \rightarrow \oplus_i  H^{*-\lambda^{\kappa}_i}_{G}\! (B^{\kappa}_i,\! \CF) \xrightarrow{\varphi} H_{G }^*(f^{-1}(\!(-\infty, \kappa + \epsilon]),\! \CF) \rightarrow H_{G}^*(f^{-1}(\!(-\infty, \kappa - \epsilon]),\!\CF) \rightarrow 0,\]
	such that the composition of the injection $\varphi$ with the restriction to $C^\kappa_s$ is the sum of the products with the Euler classes $E^\kappa_i \in H^{\lambda^{\kappa}_i}_{G}(B^{\kappa}_i,\CF) \simeq H^{\lambda^{\kappa}_i}_{\g \oplus \R \xi}(B^{\kappa}_i)  $ of the negative normal bundles of the $B^\kappa_i$.	
\end{lemma}
\begin{proof}
	The proof works analogous to the corresponding part of the proof of Theorem \ref{theorem main}.
\end{proof}
	The symplectic analogue of the following proposition was remarked after Theorem 3.2 in \cite{tolman2003kernel}.
\begin{proposition}\label{prop injection}
	Let $f= \Psi^{X_{s+1}}|_{Y_s} $ or $f= - \Psi^{X_{s+1}}|_{Y_s} $. For every regular value $a$ of $f$, the restriction
	\[ H_{G }^*(f^{-1}(\!(-\infty,a]), \CF) \rightarrow H_{G }^*(f^{-1}(\!(-\infty, a])\cap C_s, \CF) \]
	is injective.
\end{proposition}
\begin{proof}
	This proposition is proved by induction on the number $k$ of critical values below $a$.
	For $k=1$, the Morse-Bott property of $f$ yields a bijection by Proposition \ref{prop homotopic same cohom} since the homotopy type does not change before crossing another critical value. Now, suppose the claim holds for $k$. Let $a$ be a regular value of $f$ with $k+1$ critical values below it and let $\delta >0$ such that $a-\delta $ is regular and such that there are $k$ critical values below $a-\delta$. Lemma \ref{lem ses} then yields that the restriction of $H_{G }^*(f^{-1}(\!(-\infty,a]), \CF)$ to $H_{G }^*(f^{-1}(\!(-\infty,a-\delta]), \CF)$ is surjective and we obtain the commuting diagram
	\[
\begin{xy}
 \xymatrix@-5pt{
    H_{G}^*(f^{-1}(\!(-\infty,a]), \CF) \ar@{->>}[r] \ar[d] & H_{G}^*(f^{-1}(\!(-\infty,a-\delta]), \CF)\ar@{^{(}->}[d] \\ 
  H_{G}^*(f^{-1}(\!(-\infty,a])\cap C_s, \CF) \ar[r] & H_{G}^*(f^{-1}(\!(-\infty,a-\delta])\cap C_s, \CF)
 }
\end{xy}
.\]
Suppose $\sigma \in H_{G}^*(f^{-1}(\!(-\infty,a]), \CF)$ such that $\sigma |_{f^{-1}(\!(-\infty,a])\cap C_s} = 0$. In particular, we have $ \sigma |_{f^{-1}(\!(-\infty,a-\delta])\cap C_s} = 0$, so by our induction's assumption, it is $ \sigma |_{f^{-1}(\!(-\infty,a-\delta])} = 0$. I.e., $\sigma$ lies in the kernel of the restriction to $f^{-1}((-\infty,a-\delta])$. By Lemmata \ref{lem ses} and \ref{lem wedge E injective}, the restriction of this kernel to $C_s^\kappa$ is injective. But $\sigma|_{C_s^\kappa}=0$, hence, $\sigma=0$.
\end{proof}
Since we assumed 0 to be a regular value, we obtain as a direct consequence
\begin{corollary}\label{cor injection}
	The following restrictions are injective:
	\begin{align*}
		H_{G}^*(Y_s^\pm, \CF)&\rightarrow H_{G}^*(Y_s^\pm \cap C_s, \CF)\\
		H_{G}^*(Y_s, \CF)&\rightarrow H_{G}^*(C_s, \CF).
	\end{align*}
	In particular, $H_G^*(M,\CF)\rightarrow H_G^*(\mathit{Crit}(\Psi), \CF)$ and $H_G^*(M^\pm,\CF)\rightarrow H_G^*(\mathit{Crit}(\Psi)\cap M^\pm, \CF)$ are injective.
\end{corollary}
\begin{corollary}\label{cor ker j pm}
	Set $K_s^\pm := \{ \sigma \in H_{G}^*(Y_s, \CF) \mid \sigma|_{Y_s^\pm \cap C_s} = 0 \}$. Then we have  $K_s^\pm = \operatorname{ker}((j_s^\pm)^*)$, where $j_s^\pm: Y_s^\pm \to Y_s$ denotes the inclusion.
\end{corollary}
\begin{proof}
	Obviously $\operatorname{ker}((j_s^\pm)^*) \subset K_s^\pm$. Corollary \ref{cor injection} yields the reverse inclusion.
\end{proof}
\begin{remark}
	We also know that the induced maps $(j_s^\pm)^* :$ $H^*_G(Y_s, \CF)$ $\to$ \linebreak $H^*_G(Y_s^\pm, \CF) $ in equivariant basic cohomology 	
	are surjective. 
	
\begin{proof}
	We know from the proof of Theorem \ref{theorem main} that $(\iota_s^\pm)^*$ is surjective. So for every $\omega^\pm \in H^*_G(Y_s^\pm, \CF)$ there exists $\omega^\mp \in H^*_G(Y_s^\mp, \CF)$ such that $(\iota_s^\pm)^*\omega^\pm = (\iota_s^\mp)^* \omega^\mp$. Exactness of Sequence \eqref{eq MVS i j} yields that $\omega^+ + \omega^- \in \operatorname{ker}((\iota_s^+)^* - (\iota_s^-)^*) = \operatorname{im}((j_s^+)^* + (j_s^-)^*)$, hence, there exists $\sigma \in H^*_G(Y_s, \CF) : \omega^\pm = (j_s^\pm)^* \sigma$.
\end{proof}
\end{remark}

As a consequence of the previous corollary, we then obtain
the following, which is the contact analogue of \cite[Theorem 2]{tolman2003kernel}.
\begin{theorem}\label{thm kernel}
	Let $G=S^1$ and set $C^\pm := \mathrm{Crit}(\Psi) \cap M^\pm$, $K^\pm = \{\sigma \in H_{G}^*(M, \CF) \mid \sigma|_{C^\pm}=0 \}$.
	The kernel $K$ of the Kirwan map $H_{G}^*(M, \CF)\rightarrow H_{G}^*(\Psi^{-1}(0), \CF)$ is given by
	\[K=  K^+ \oplus  K^-.\]
\end{theorem}
\begin{proof}
	By Corollary \ref{cor ker j pm}, $K^\pm = \ker (j^\pm)^*$. It follows that $K^\pm \subset 
	\ker(i^\pm)^*\circ(j^\pm)^*$, so $K^+\oplus K^-$ lies in the kernel of the Kirwan map. For the reverse inclusion, consider the Mayer-Vietoris sequence (see Proposition \ref{prop MVS}) for $(M, M^+, M^-)$ (or, more precisely, of the two open sets $\{x \in M \mid \pm \Psi(x) < \epsilon\}$ for sufficiently small $\epsilon >0$ which, by the Morse-Bott property of $\Psi$, are of the same $G\times T$-homotopy type as $M^\pm$.). In \eqref{eq MVS i j}, we saw that it actually consists of the short exact sequences
	\begin{align*}
	  0 \rightarrow H_{G}^*(M,\CF) \stackrel{(j^+)^* \oplus (j^-)^*}{\rightarrow} \! \! \! H_{G}^*(M^+,\CF)\oplus H_{G}^*(M^-,\CF)\stackrel{(i^+)^* - (i^-)^*}{\to} \! \! H_{G}^*(\Psi^{-1}(0),\CF)\rightarrow 0.
	\end{align*}
	Now, suppose $\eta$ lies in the kernel of the Kirwan map, i.e., $(i^\pm)^*(j^\pm)^*\eta = 0$. This means, however, that $ (j^+)^*\eta \oplus 0$ and $ 0 \oplus(j^-)^*\eta $ lie in the kernel of $(i^+)^* - (i^-)^*$. By exactness of above sequence, there exist $\eta^\pm \in H_{G}^*(M,\CF) $ such that $(j^+)^* \oplus (j^-)^*(\eta^\pm)= (j^\pm)^*\eta$, in particular, $\eta^\pm \in K^\pm$ by Corollary \ref{cor ker j pm}. Then $(j^+)^* \oplus (j^-)^*(\eta^+ + \eta^-)=(j^+)^* \oplus (j^-)^*(\eta)$. But $(j^+)^* \oplus (j^-)^*$ is injective because the sequence is exact, thus $\eta = \eta^+ + \eta^- \in K^+ \oplus K^-$.
\end{proof}
\begin{remark}
	We remark that Theorem \ref{thm kernel} can also be proved similarly to the proof presented in \cite[Theorem 2]{tolman2003kernel}. In the symplectic case, this proof generalizes to the setting of the action of higher rank tori, where Morse-Bott theory of the norm square of the symplectic moment map is applied. We believe that, in the contact setting, an analogous description of the kernel holds for the action of tori of higher rank, as well. This is work in progress.
\end{remark}

\begin{example}
	Let us continue the example presented in Section \ref{subsection ex sphere}, with $w=(1,...,1,w_n)$. Then \[H_G(M, \CF)= \R[u,s]/\langle e \rangle = \R[u,s]/\langle (u+s)^n (-u+w_ns) \rangle.\] We have $M^+= \{ z \in S^{2n+1} \mid |z_n|^2 \leq \tfrac{1}{2} \}$ and $M^-= \{ z \in S^{2n+1} \mid |z_n|^2 \geq \tfrac{1}{2} \}$
	so that $C^+=\mathit{Crit}(\Psi) \cap M^+ = S^{2n-1}\times \{0\}$ and $C^- = \mathit{Crit}(\Psi)\cap M^- = \{0\} \times S^1$. Making use of homotopy equivalences, Lemma \ref{lem ses} with $\Psi$ yields that we have a short exact sequence and a commuting diagram
	\[ \begin{xy}
 \xymatrix@-13pt{
    0 \ar[r] & H_{G}^{*-\lambda^+}(C^{+},\CF) \ar[dr]_{\cdot e(\nu^-C^+)} \ar[r] & H_{G}^*(M, \CF) \ar[d]_{\cong}^{i_{C^+}^*} \ar[r]  & H_{G}^*(C^-,\CF) \ar[r] &0,\\ 
    & & H^{*}_{G}(C^+,\CF) 
    }
\end{xy}
	\]
	where $e(\nu^-C^+)$ denotes the equivariant basic Euler class of the negative normal bundle $\nu^-C^+$ of $C^+$ and $\lambda^+$ the rank of $\nu^-C^+$ and $i_{C^+}:C^+\rightarrow M$ denotes the inclusion. Similarly, with $-\Psi$, we obtain a short exact sequence
	\[ \begin{xy}
 \xymatrix@-13pt{
    0 \ar[r] & H_{G}^{*-\lambda^-}(C^{-},\CF) \ar[dr]_{\cdot e(\nu^+C^-)} \ar[r] & H_{G}^*(M, \CF) \ar[d]_{\cong}^{i_{C^-}^*} \ar[r]  & H_{G}^*(C^+,\CF) \ar[r] &0,\\ 
    & & H^{*}_{G}(C^-,\CF) 
    }
\end{xy}
	\]
	where $e(\nu^+C^-)$ denotes the equivariant basic Euler class of the positive normal bundle $\nu^+C^-$ of $C^-$ and $\lambda^-$ the rank of $\nu^+C^-$ and $i_{C^-}:C^-\rightarrow M$ denotes the inclusion.
	Note that the standard Riemannian metric $g$ on $S^{2n+1}$ is $S^1\times T$-invariant. The normal bundles of $C^+$ and $C^-$ are then given by $\nu C^+ = (\{0\}\times \mathbb{C})\times C^+= \operatorname{span}\{\partial_{x_n}, \partial_{y_n}\}$ and $\nu C^- = (\mathbb{C}^n \times \{0\})\times C_2=\operatorname{span}\{\partial_{x_j}, \partial_{y_j} \mid j=0, ..., n-1\}$, respectively, where we used the notation $z_j = x_j + iy_j$. In these bases, the Hessian $H$ of $\Psi$ computes as
	\[ H|_{\nu C^+} = \left(\begin{smallmatrix} -2(1+w_n) & 0 \\ 0 & -2(1+w_n) \end{smallmatrix}\right), \quad H|_{\nu C^-} = \left(\begin{smallmatrix} \tfrac{2(1+w_n)}{w_n^2} & 0 & & &0\\ 
				      0 & \tfrac{2(1+w_n)}{w_n^2} & & & \\
				       & & \ddots & &\\
				       & & & \tfrac{2(1+w_n)}{w_n^2} & 0 \\ 
				       0 & & & 0 & \tfrac{2(1+w_n)}{w_n^2}\end{smallmatrix}\right).\]
	Since $w_n>0$, it follows that $\nu^- C^+=\nu C^+$ and $\nu^+ C^-=\nu C^-$. Similarly to the computation in Section \ref{subsection ex sphere}, we compute $i_{C^+}^*=(s \mapsto d\alpha - u)$, $i_{C^-}^*=(s\mapsto \tfrac{u}{w_n})$ and the Euler classes 
	\begin{align*} 
		e(\nu^-C^+)&=\frac{1}{2\pi}(-u+sw_n) = \frac{1}{2\pi}(w_nd\alpha - (1+w_n)u), \\ e(\nu^+C^-)&=\frac{1}{(2\pi)^n}(u+s)^n = \left(\frac{1+\tfrac{1}{w_n}}{2\pi}\right)^n u^n.
	\end{align*}
 
	Since the inclusion $C^+ \cup C^- \to M$ induces an injective map in equivariant basic cohomology by Corollary \ref{cor injection}, $K^\pm$ consists exactly of those classes that vanish when restricted to $C^\pm$ and that are a multiple of $e(\nu C^\mp)$ when restricted to $C^\mp$. Again making use of injectivity, we get
	\begin{align*}
	K^+&=\R[u,s]\cdot (u+s)^n/\langle e \rangle  \ \quad \subset \R[u,s]/\langle e \rangle \text{ and }\\
	K^-&=\R[u,s]\cdot (-u+sw_n)/\langle e\rangle \subset \R[u,s]/\langle e \rangle.
	\end{align*}
	Indeed, we see that 
	\begin{align*}
		 H_G(M, \CF) \left/\left(K^++K^-\right)\right. &\cong \ \R[u,s]\left/\left(\R[u,s]\cdot (u+s)^n + \R[u,s]\cdot (-u+sw_n)\right)\right. \\
		 & \cong \ \R[u]\left/{\langle u^n \rangle}\right.\\
		 &\cong H_G(\Psi^{-1}(0), \CF).
	\end{align*}
\end{example}

\section{Equivariant Formality}\label{sec formality}
Another well-known result concerning the equivariant cohomology of a symplectic manifold is the \emph{formality} of Hamiltonian actions of compact connected Lie groups $H$ on compact symplectic manifolds $N$, namely that $H_H(N)$ is a free $S(\mathfrak{h}^* )$-module (cf. \cite[Proposition~5.8]{kirwan1984cohomology}). In this section, we will show that this property also holds for torus actions on $K$-contact manifolds if we consider the basic setting. For a study of equivariantly formal actions in the setting of equivariant basic cohomology of transverse actions, the reader is referred to \cite[Section~3.6]{goertsches2010equivariant}.

\begin{definition}
	Let $(N,\mathcal{E})$ be any foliated manifold, acted on by a torus $H$ s.t. $\Omega(N,\mathcal{E})$ is an $H^*$-algebra. The $H$-action on $(N,\mathcal{E})$ is called \emph{equivariantly formal}, if $H^*_H(N,\mathcal{E})$ is a free $S(\mathfrak{ h}^*)$-module.
\end{definition}

Recall that we considered a torus $G$, that acts on a $K$-contact manifold $(M, \alpha)$, leaving $\alpha$ invariant, in such a way that the contact moment map $\Psi$ has 0 as a regular value. We work with a basis $(X_i)$ of $\g$ according to Proposition \ref{prop basis K-contact} and we, again, denote the foliation induced by the Reeb vector field $\xi$ with $\CF$ and the 1-dimensional $G\times T$-orbits, i.e., the critical points of $\Psi^{X_1}$ and $\Psi$, by $C$, where $T$ denotes the closure of $\{\psi_t\}$.

\begin{lemma}\label{lem C eq formal}
	The $G$-action on $(C, \CF)$ is equivariantly formal. More precisely, we have
	\[H_G^*(C, \CF) \simeq S(\g^*)\otimes H^*(C, \CF).\]
\end{lemma}
\begin{proof}
	We have $C=\mathit{Crit}(\Psi)= \{x \in M \mid \widetilde{\g}_x = \g \}$ by Proposition \ref{prop basis K-contact}, \textit{(v)}. Lemma \ref{lem only 1 g tilde} with $\mathfrak{k}= \{0\}$ yields the claim.
\end{proof}

\begin{proposition}\label{prop M formal}
	The $G$-action on $(M,\CF)$ is equivariantly formal.
\end{proposition}
\begin{proof}
	Consider $X=X_1 \in \g$ as in Proposition \ref{prop basis K-contact}. By Lemma \ref{lem C eq formal}, the $G$-action on $C$ is equivariantly formal. Let $\kappa_1< ...< \kappa_m$ be the critical values of $\Psi^{X}$ and denote with $B^{\kappa_{j}}_1, ..., B^{\kappa_{j}}_{i_j}$ the connected components of the critical set $C$ at level $\kappa_{j}$ and with $\lambda^{\kappa_{j}}_i$ the indices of the non-degenerate critical submanifolds $B^{\kappa_j}_i$ with respect to  $\operatorname{Hess}(\Psi^X)$. Since $\Psi^X$ is a Morse-Bott function, we obtain (similar to the proof of Theorem \ref{theorem main}) the long exact sequences
	\[ \cdots \rightarrow \oplus_i H^{*-\lambda^{\kappa_j}_i}_{G}(B^{\kappa_j}_i, \CF)  \rightarrow H_{G}^*(M^{\kappa_j + \epsilon_j}, \CF)\rightarrow H_{G}^*(M^{\kappa_j - \epsilon_j}, \CF) \rightarrow \cdots \]
	If $H_{G}^*(M^{\kappa_j + \epsilon_j}, \CF) = H_{G}^*(M^{\kappa_{j+1} - \epsilon_{j+1}}, \CF)$ and $\oplus_i H^{*-\lambda^{\kappa_j}_i}_{G}(B^{\kappa_j}_i, \CF)$ are free $S(\g^*)$-modules, then, by exactness, $H_{G}^*(M^{\kappa_{j+1} + \epsilon_{j+1}}, \CF)$ has to be free, as well. Hence, induction on $j$ yields that $H^*_G(M,\CF)=H_{G}^*(M^{\kappa_m + \epsilon_m}, \CF)$ is a free $S(\g^*)$-module.
\end{proof}


\begin{thebibliography}{GGK02}

\bibitem[AB83]{atiyah1983yang}
M.F. Atiyah and R.~Bott.
\newblock {The Yang-Mills equations over Riemann surfaces}.
\newblock {\em Philos. Trans. R. Soc. London Series A}, pages 523--615, 1983.

\bibitem[BG08]{boyer2008sasakian}
C.~P. Boyer and K.~Galicki.
\newblock {\em Sasakian geometry}.
\newblock Oxford University Press, 2008.

\bibitem[BL10]{baird2010topology}
T.~Baird and Y.~Lin.
\newblock Topology of generalized complex quotients.
\newblock {\em J. Geom. Phys.}, 60(10):1539--1557, 2010.

\bibitem[Bla76]{blair1976contact}
D.~Blair.
\newblock {\em Contact manifolds in Riemannian geometry}, Vol. 509 of {Lecture
  Notes in Mathematics}.
\newblock Springer, 1976.

\bibitem[BT13]{bott2013differential}
R.~Bott and L.~W. Tu.
\newblock {\em Differential forms in algebraic topology}.
\newblock Springer, 2013.

\bibitem[Car50]{cartan1950}
H.~Cartan.
\newblock La transgression dans un groupe de lie et dans un espace fibr\'e
  principal.
\newblock {\em Colloque de Topologie Bruxelles}, 57--71, 1950.

\bibitem[CdS01]{cannas2001lectures}
A.~Cannas~da Silva.
\newblock {\em Lectures on symplectic geometry}, Vol. 1764 of {\em Lecture
  Notes in Mathematics}.
\newblock Springer, 2001.

\bibitem[CF17]{CF17loc}
L.~Casselmann and J.~Fisher.
\newblock {Localization for $K$-contact manifolds}.
\newblock {\href{https://arxiv.org/abs/1703.00333}{arXiv preprint: 1703.00333}}, 2017.

\bibitem[Duf83]{duflot1983smooth}
J.~Duflot.
\newblock Smooth toral actions.
\newblock {\em Topology}, 22(3):253--265, 1983.

\bibitem[GGK02]{ggk2002moment}
V.~L. Ginzburg, V.~Guillemin, and Y.~Karshon.
\newblock {\em Moment maps, Cobordisms, and Hamiltonian group actions},
  Vol.~98 of {\em Math. Surveys and Monographs}.
\newblock Amer. Math. Soc., 2002.

\bibitem[GNT12]{GNTequivariant}
O.~Goertsches, H.~Nozawa, and D.~T{\"o}ben.
\newblock {Equivariant cohomology of K-contact manifolds}.
\newblock {\em Math. Ann.}, 354:1555--1582, 2012.

\bibitem[GNT17]{GNTlocalization}
O.~Goertsches, H.~Nozawa, and D.~T{\"o}ben.
\newblock {Localization of Chern-Simons type invariants of Riemannian
  foliations}.
\newblock {\em Israel J. Math}, 222(2):867--920, 2017.


\bibitem[Gol02]{goldin2002effective}
R.~F. Goldin.
\newblock {An effective algorithm for the cohomology ring of symplectic reductions}.
\newblock {\em Geom. Funct. Anal.}, 12(3):567--583, 2002.  
  
\bibitem[GT10]{goertsches2010torus}
O.~Goertsches and D.~T{\"o}ben.
\newblock {Torus actions whose equivariant cohomology is Cohen--Macaulay}.
\newblock {\em J. Topology}, 3(4):819--846, 2010.

\bibitem[GT16]{goertsches2010equivariant}
O.~Goertsches and D.~T{\"o}ben.
\newblock Equivariant basic cohomology of Riemannian foliations.
\newblock {\em J. Reine Angew. Math. (Crelles
  J.)}, 2016.
  
\bibitem[GS99]{guillemin1999supersymmetry}
V.~W. Guillemin and S.~Sternberg.
\newblock {\em Supersymmetry and Equivariant de Rham Theory},
\newblock Springer, 1999.

\bibitem[Kir84]{kirwan1984cohomology}
F.~Kirwan.
\newblock {\em Cohomology of Quotients in Symplectic and Algebraic Geometry},
\newblock Princeton Univ. Press, 1984.

\bibitem[Kob58]{kobayashi1958fixed}
S.~Kobayashi.
\newblock Fixed points of isometries.
\newblock {\em Nagoya Math. J}, 13(1):63--68, 1958.

\bibitem[Ler04]{lerman04question}
E.~Lerman.
\newblock Question 3.1.
\newblock In {\em AIM Workshop on Moment Maps and Surjectivity in Various
  Geometries: Conjectures and Further Questions}, 2004.
  \newblock \href{http://aimath.org/WWN/momentmaps/momentmaps.pdf}{http://aimath.org/WWN/momentmaps/momentmaps.pdf} Accessed 21 June 2013.

\bibitem[Mol88]{molino1988riemannian}
P.~Molino.
\newblock {\em Riemannian Foliations}.
\newblock Birkh\"auser Boston Inc., 1988.

\bibitem[PV07]{paradanverne2007relative}
P.-E.~Paradan and M.~Vergne.
\newblock Equivariant relative Thom forms and Chern characters.
\newblock {\em arXiv preprint arXiv:0711.3898}, 2007.

\bibitem[Rei59]{reinhart1959harmonic}
B.~L. Reinhart.
\newblock Harmonic integrals on foliated manifolds.
\newblock {\em Amer. J. Math.}, 81(2):529--536, 1959.

\bibitem[Ruk95]{rukimbira1995topology}
P.~Rukimbira.
\newblock {Topology and closed characteristics of K-contact manifolds.}
\newblock {\em Bull. Belg. Math. Soc.}, 2(3):349--356,
  1995.

\bibitem[Ruk99]{rukimbira1999}
P.~Rukimbira.
\newblock {On K-contact manifolds with minimal number of closed
  characteristics}.
\newblock {\em Proc. Amer. Math. Soc.}, 127:3345--3352, 1999.

\bibitem[TW03]{tolman2003kernel}
S.~Tolman and J.~Weitsman.
\newblock The cohomology rings of symplectic quotients.
\newblock {\em Comm. Anal. Geom.}, 11(4):751--774, 2003.

\bibitem[Was69]{wasserman1969equivariant}
A.~G. Wasserman.
\newblock Equivariant differential topology.
\newblock {\em Topology}, 8(2):127--150, 1969.

\bibitem[Wei97]{weibel1997introduction}
C.~A. Weibel.
\newblock {\em An introduction to homological algebra}. Vol. 38 of {\em Cambridge studies in advanced mathematics}.
\newblock Cambridge Univ. Press, 1997.

\end{thebibliography}
\end{document}